\documentclass[11pt]{amsart}
\usepackage[T1]{fontenc}
\usepackage[latin9]{inputenc}
\usepackage{geometry}
\usepackage{amsthm}
\usepackage{setspace}
\usepackage{amssymb}
\usepackage{graphicx}
\usepackage{epstopdf}
\usepackage{diagrams}
\usepackage{verbatim}
\usepackage[active]{srcltx}  
\numberwithin{equation}{section} 
\numberwithin{figure}{section} 
\theoremstyle{plain}

\DeclareMathOperator{\dist}{dist}
\DeclareMathOperator{\diam}{diam}

\DeclareMathOperator{\supp}{supp}
\DeclareMathOperator{\vol}{vol}
\DeclareMathOperator{\essinf}{essinf}
\DeclareMathOperator{\id}{id}

\newcommand{\e}{\varepsilon}
\newcommand{\NN}{\mathbb{N}}
\newcommand{\RR}{\mathbb{R}}

\newtheorem{thm}{Theorem}[section]
  \theoremstyle{plain}
  \newtheorem{conjecture}[thm]{Conjecture}
  \theoremstyle{plain}
  \newtheorem{lem}[thm]{Lemma}
  \theoremstyle{plain}
  \newtheorem{prop}[thm]{Proposition}
  \theoremstyle{plain}
  \newtheorem{cor}[thm]{Corollary}
  \theoremstyle{definition}
  \newtheorem{defn}[thm]{Definition}

\begin{document}

\title{Local entropy averages and projections of fractal measures}

\author{Michael Hochman}
\address{Michael Hochman\\
         Princeton university\\
         Fine Hall\\
         Washington Rd.\\
         Princeton, NJ 08544\\
         USA}
\email{hochman@math.princeton.edu}

\author{Pablo Shmerkin}
\address{Pablo Shmerkin\\
         University of Manchester. Oxford Road\\
         School of Mathematics. Alan Turing Building\\
         Manchester M13 9PL\\
         UK }
\email{Pablo.Shmerkin@manchester.ac.uk}

\thanks{This research was partially supported by the NSF under agreement No. DMS-0635607. M.H. supported by NSF grant 0901534. P.S. acknowledges support from EPSRC grant EP/E050441/1 and the University of Manchester.}


\date{\today}

\begin{abstract}
We show that for families of measures on Euclidean space which satisfy an ergodic-theoretic form of ``self-similarity'' under the operation of re-scaling, the dimension of linear images of the measure behaves in a semi-continuous way. We apply this to prove the following conjecture of Furstenberg: if $X,Y\subseteq [0,1]$ are closed and invariant, respectively, under $\times m\bmod 1$ and $\times n\bmod 1$, where $m,n$ are not powers of the same integer, then, for any $t\neq0$, \[
\dim(X+t Y)=\min\{1,\dim X+\dim Y\}.\]
A similar result holds for invariant measures, and gives a simple proof of the Rudolph-Johnson theorem. Our methods also apply to many other classes of conformal fractals and measures. As another application, we extend and unify results of Peres, Shmerkin and Nazarov, and of Moreira, concerning projections of products self-similar measures and Gibbs measures on regular Cantor sets. We show that under natural irreducibility assumptions on the maps in the IFS, the image measure has the maximal possible dimension under any linear projection other than the coordinate projections. We also present applications to Bernoulli convolutions and to the images of fractal measures under differentiable maps.
\end{abstract}

\maketitle

\markboth{M. Hochman and P. Shmerkin}{Local entropy averages and projections}

\section{\label{sec:Introduction}Introduction }

\subsection{\label{sub:Background}Background and history}

Let $\dim$ denote the Hausdorff dimension of a set and let $\Pi_{d,k}$
denote the space of orthogonal projections from $\mathbb{R}^{d}$
to $k$-dimensional subspaces, with the natural measure. Then it is
a classical fact, due in various versions to Marstrand, Mattila and
others, that for any Borel set $X\subseteq\RR^{d}$, almost every $\pi\in\Pi_{d,k}$
satisfies \begin{equation}
\dim(\pi X)=\min(k,\dim X).\label{eq:projections-no-not-increase-dimension}\end{equation}
Indeed, the right hand side is a trivial upper bound: Lipschitz maps
cannot increase dimension, so $\dim\pi X\leq\dim X$, and $\pi X$
is a subset of a $k$-dimensional subspace, hence $\dim\pi X\leq k$.
Since this equality holds almost everywhere, we shall use the term
\emph{exceptional} for projections for which equality fails, and call
the right hand side the \emph{expected} dimension.

While \eqref{eq:projections-no-not-increase-dimension} tells us what
happens for a typical projection, it is far more difficult to analyze
the image of even the simplest fractals under individual projections. See Kenyon \cite{Kenyon97} for a particularly simple and frustrating example.

There are, however, a number of well-known conjectures to the effect
that, for certain sets of combinatorial, arithmetic or dynamical origin,
phenomena which hold typically in the general setting should, for
these special sets, always hold, except in the presence of some evident obstruction.
The present work was motivated by a conjecture of this kind concerning
projections of product sets whose marginals are invariant under arithmetically
{}``independent'' dynamics.

Denote the $m$-fold map of the unit interval by $T_{m}:x\mapsto mx\bmod1$
and let $\pi_{x},\pi_{y}\in\Pi_{2,1}$ denote the coordinate projections
onto the axes.

\begin{conjecture} [Furstenberg] \label{con:Furstenberg-conjecture}Let
$X,Y\subseteq[0,1]$ be closed sets which are invariant under $T_{2}$
and $T_{3}$, respectively. Then \[
\dim\pi(X\times Y)=\min\{1,\dim(X\times Y)\}\]
 for any $\pi\in\Pi_{2,1}\setminus\{\pi_{x},\pi_{y}\}$. \end{conjecture}

In the situation above it is evident that $\pi_{x},\pi_{y}$ are exceptions,
since they map $X\times Y$ to $X$ or $Y$, respectively, and a drop
in dimension is to be expected.

Note that this conjecture can also be formulated
as a result on sumsets:%
\footnote{In the sumset formulation we relied on the identity \[
\dim(X\times Y)=\dim(X)+\dim(Y)\]
This holds in the present case because $X$ has coinciding Hausdorff
and box dimension (see e.g. \cite[Theorem 5.1]{Furstenberg08}); in
general one only has the inequality $\dim(X\times Y)\geq \dim(X)+\dim(Y)$. %
} for $X,Y$ as above and all $s\neq0$,\[
\dim(X+sY)=\min\{1,\dim X+\dim Y\}.\]
 Here $A+B=\{a+b\,:\, a\in A\,,\, b\in B\}$.

Conjecture \ref{con:Furstenberg-conjecture} originates in the late 1960s. Although it has apparently not appeared in print, it is related to another conjecture of Furstenberg's from around the same time, which appears in \cite{Furstenberg70}:

\begin{conjecture}[Furstenberg, \cite{Furstenberg70}]\label{con:intersections}
Let
$X,Y\subseteq[0,1]$ be closed sets which are invariant under $T_{2}$
and $T_{3}$, respectively. Then for any $s,t$, $t\neq 0$,\[
   \dim X \cap (s+tY) \leq \max\{\dim X +\dim Y -1, 0\}
\]
\end{conjecture}

The relation between these conjecture is as follows. The sets $X\cap (s+tY)$ are, up to affine coordinate change, the intersections of $X\times Y$ with the fibers of the projections $\pi\in\Pi_{2,1}\setminus\{\pi_x,\pi_y\}$. Heuristically, one expects the following relation between the dimension of the image and the fibers: \begin{equation}
   \dim X\times Y = \dim \pi(X\times Y) + \sup_z\{\dim ((X\times Y)\cap \pi^{-1}(z))\} \label{eq:dimension-conservation}
\end{equation}
This is, for example, the way affine subspaces in $\mathbb{R}^d$ behave under linear maps, as do generic sub-manifolds, and if it were true then Conjectures \ref{con:Furstenberg-conjecture} and \ref{con:intersections} would be equivalent: if the image under a projection has the expected dimension then the fibers would behave as expected as well.

Equation \eqref{eq:dimension-conservation} is a very strong statement, and simple examples show that it is not generally true. For quite general sets  $A\subseteq \mathbb{R}^d$ it is known that, under a natural distribution on the $k$-dimensional subspaces which intersect $A$, almost every such subspace intersects $A$ in the expected dimension, i.e. the larger of $\dim A - k$ and $0$; see \cite[Theorem 10.11]{Mattila95}. For certain special sets $A\subseteq\mathbb{R}^2$, related in some ways to the product sets we are discussing, a stronger result was obtained by Furstenberg \cite{Furstenberg08}: for every $t\neq 0$ there are many (in the sense of dimension) values of $s$ such that $A$ intersects the line $x-ty=s$ in \emph{at least} the expected dimension. However, the uniform upper bounds needed for Conjecture \ref{con:intersections} still seems out of reach of current methods.

We refer the reader to \cite{Furstenberg70} for a more detailed discussion of Conjecture \ref{con:intersections} and some related questions.

\subsection{Iterated Function Systems}

A related circle of questions concerns projections of product sets whose marginals are attractors of iterated function systems (IFSs) on the line. Here again it is believed that, in the absence of some evident ``resonance'' between the IFSs, projections should behave as expected.

There was little progress on Conjecture \ref{con:Furstenberg-conjecture} until fairly recently. The first result of this kind is a theorem by C. G. Moreira \cite{Moreira98}\footnote{The proof in \cite{Moreira98} is incomplete, see e.g. \cite{PeresShmerkin09}} for pairs of regular IFSs, i.e. systems of $C^{1+\varepsilon}$ contractions on the line satisfying the strong separation condition (see Section \ref{sec:convolutions-Gibbs-measures}). Moreira assumes that at least one of the IFSs is strictly non-linear, i.e. cannot be conjugated to a linear one, and that a certain irrationality condition is satisfied between the IFSs. Under these hypotheses he shows that if $X,Y$ are the attractors then the sumset $X+Y$, which is the projection of $X\times Y$ under $\pi(x,y)=x+y$, has the expected dimension.

More recently Y. Peres and P. Shmerkin \cite{PeresShmerkin09} solved the problem for projections of $X\times Y$ when $X,Y$ are attractors of linear IFSs satisfying an irrationality condition, namely, that the logarithms of some pair of contraction ratios is rationally independent. This class of examples includes some special cases of \ref{con:Furstenberg-conjecture}. For example, the standard middle-third Cantor set is both $T_3$-invariant and the attractor of an IFS with contraction ratio $1/3$. With an eye to Furstenberg's conjecture, these methods can be pushed to apply to $T_{m}$-invariant subsets of $[0,1]$ which are shifts of finite type with respect to the base-$m$ coding. See also \cite{FergusonJordanShmerkin09} for an extension to some non-conformal attractors on the plane.

With regard to the question of projecting measures rather than sets,  Nazarov, Peres and Shmerkin \cite{NazarovPeresShmerkin09} recently established some results for projections of Hausdorff measure on $X\times Y$, where $X,Y$ are now attractors of linear IFSs in which, additionally, all contracting maps have the same contraction ratio (while these assumptions are quite special, in this case they also establish a stronger result using correlation dimension rather than Hausdorff dimension).

It is interesting to note that the methods of Moreira and of Peres-Shmerkin are quite different and rely heavily on their respective assumptions, i.e. strict non-linearity and linearity of the IFSs. This leaves open the case of a pair of IFSs which are both non-linearly conjugated to linear IFSs. Their methods also do not give any information about behavior of measures on regular IFSs.

Finally, similar questions may be asked about multidimensional attractors of IFSs rather than products one-dimensional ones. For the case of 2-dimensional linear IFSs Peres and Shmerkin \cite{PeresShmerkin09} showed that, assuming that the orthogonal part of the contractions include at least one irrational rotation, all projections behave as expected. Unfortunately these methods do not work for dimension $d\geq 3$, and again give no information about measures.

\subsection{\label{sub:Results}Results}

In this work we develop a method for bounding from below the dimension
of projections of measures which exhibit certain statistical self-similarity.
Before describing the general result, we summarize our main applications.

The first is a resolution of Conjecture \ref{con:Furstenberg-conjecture} in its full generality. In fact, we establish
a stronger statement concerning invariant measures. Recall that for
a probability measure $\mu$ on a metric space, the lower Hausdorff dimension $\dim_*\mu$ is defined as
\[
\dim_*\mu = \inf\{\dim(A):\mu(A)>0\}.
\]
Also, write $\dim\mu=\alpha$ to indicate that \[
  \lim_{r\downarrow 0}\frac{\mu(B_r(x))}{\log r}=\alpha\quad\textrm{for } \mu\text{-a.e. } x.
\]
In this case $\alpha$ is the exact dimension of $\mu$ and $\dim_*\mu=\dim\mu$, but note that $\dim\mu$ is not always defined. See Section \ref{sec:Dimension-and-entropy} for a discussion of dimension.

\begin{thm} \label{thm:projection-of-product-invariant-measures}Let
$\mu,\nu$ be Borel probability measures on $[0,1]$ which are invariant
under $T_{m},T_{n}$, respectively, and $m,n$ are not powers of the
same integer. Then for every $\pi\in\Pi_{2,1}\setminus\{\pi_{x},\pi_{y}\}$,
\[
\dim_*\pi(\mu\times\nu)=\min\{1,\dim_*(\mu\times\nu)\}
\]
If $\mu,\nu$ are exact dimensional then the above holds for $\dim$ instead of $\dim_*$.
 \end{thm}

Note that both the conjecture above and the theorem hold trivially
in dimension zero. From the theorem one proves the conjecture using the variational principle
to relate the dimension of sets and measures; see Section \ref{sub:Proof-of-topological-conj}.

Theorem \ref{thm:projection-of-product-invariant-measures} also leads
to a very short proof of the Rudolph-Johnson theorem (see Section \ref{sub:Rudolph-Johnson-theorem}): If $m,n$ are not powers of the same integer, $\mu$ is a probability measure
on $[0,1]$ invariant under both $T_{m}$ and $T_{n}$, and all ergodic components have positive entropy for one (equivalently both) of the maps, then $\mu=$Lebesgue measure. Unfortunately, neither this proof nor our methods provide a hint on
how to approach the long-standing conjecture about the entropy zero
case.

For the next result we require some notation. For general definitions regarding IFSs see Sections \ref{sec:self-similar-measures}, \ref{sec:convolutions-Gibbs-measures}. Given a contracting smooth map $f$ on $[0,1]$, we let 
\begin{equation} \label{eq:def-lambda}
\lambda(f) = -\log(f'(x)), \text{where }x\text{ is the fixed point of }f.
\end{equation}
Furthermore, if $\mathcal{I}= \{ f_i :i\in\Lambda\}$ is a regular IFS (see Section \ref{sec:convolutions-Gibbs-measures} for the definition), we let
\[
L(\mathcal{I}) = \{  \lambda(f_{x_1}\circ\cdots \circ f_{x_n}): n\in\NN, \,x_1,\ldots, x_n\in \Lambda\}.
\]

\begin{thm} \label{thm:convolutions-of-Gibbs-measures}  Let $\mathcal{I}^{(i)}=\{ f_{j}^{(i)}:j\in\Lambda_{i}\} $, $i=1,\ldots, d$
be regular IFSs with attractor $X_{i}$. Suppose the following holds:

\textbf{Minimality assumption}. The set  $L(\mathcal{I}^{(1)}) \times\cdots \times L(\mathcal{I}^{(d)})$ is dense in the quotient space $(\mathbb{R}^d,+)/\Delta$, where $\Delta$ is the diagonal subgroup of $\mathbb{R}^d$.

Then for any globally supported Gibbs measures $\mu_{i}$ on
$X_{i}$ corresponding to arbitrary H\"{o}lder potentials, and for any projection $\pi(x)=\sum_i t_i x_i$ with all $t_i$ nonzero,
\[
\dim\left(\pi(\mu_{1}\times\cdots \times \mu_{d})\right)=\min(1,\dim(\mu_{1})+\ldots+\dim(\mu_{d})).\]
\end{thm}

A classical result of R. Bowen shows that Hausdorff measure on a regular Cantor set is equivalent to a Gibbs measure for an appropriate potential. We therefore have

\begin{cor}If $X_i$ are attractors of IFSs satisfying the hypotheses of the theorem, then for any $\pi\in\Pi_{d,k}$,
\begin{equation}
\dim\pi(X_1\times \cdots\times X_d)=\min\{1,\dim X_1+\ldots +\dim X_d\} \label{eq:projections-of-sets}\end{equation}
\end{cor}

The minimality condition in Theorem \ref{thm:convolutions-of-Gibbs-measures} is satisfied, for example, when there are $d$ rationally independent numbers among the numbers $\lambda(f^{(i)}_j)$. Thus Theorem \ref{thm:convolutions-of-Gibbs-measures} (and its corollary) generalizes and extends the aforementioned results of Moreira \cite{Moreira98}, Peres-Shmerkin \cite{PeresShmerkin09} and Nazarov-Peres-Shmerkin \cite{NazarovPeresShmerkin09}. We note that Theorem \ref{thm:convolutions-of-Gibbs-measures}  does not make any assumptions about the linear or non-linear nature of the IFSs, and neither does the proof, which provides a unified treatment of the known cases. We also remark that Moreira [private communication] has shown that, for $d=2$, the minimality assumption holds automatically when one of the IFS is not conjugated to a linear IFS. In the linear case, however, the minimality assumption may fail to hold and is necessary; see \cite{PeresShmerkin09} for a discussion.

For self-similar sets and measures in $\mathbb{R}^d$ (see Section \ref{sec:self-similar-measures}), we have:

\begin{thm} \label{thm:projections-of-self-similar-measures-with-irrational-rotation}
Let $\{f_{i}:i\in\Lambda\}$ be an iterated function system on $\mathbb{R}^d$ with the strong separation condition consisting
of similarities, and $\mu$ a self-similar measure on its attractor. Let $O_i$ denote the orthogonal part of the similarity $f_i$ and suppose that:

\textbf{Minimality assumption}. The action (by right composition) of
the semigroup generated by $O_{i}$ on $\Pi_{d,k}$ is topologically
minimal, i.e. for some (equivalently any)  $\pi\in\Pi_{d,k}$ the orbit
\[
\{\pi O_{i_1}\cdots O_{i_k}:i_1,\ldots, i_k\in\Lambda\}
\]
is dense in $\Pi_{d,k}$.

Then for every $C^{1}$  map $g:\supp\mu\rightarrow\RR^k$ without singular points, $g\mu$ is exact dimensional  and
\[
\dim(g\mu)=\min(k,\dim\mu).
\]
 \end{thm}

It is well known that, under the strong separation condition, for a self-similar set of dimension $\alpha$, the $\alpha$-dimensional Hausdorff measure on $X$ is equivalent to a self-similar measure. Therefore the theorem implies a version for sets:

\begin{cor}If $X$ is an attractor of an IFS satisfying the hypotheses of the theorem above, then for every $g\in C^1(X)$ without singular points, \[
  \dim g(X)=\min\{k,\dim X\}.
\]
\end{cor}

Finally, our methods also apply to certain problems involving non-smooth
maps. Recall that the biased Bernoulli convolution with contraction $0<t<1$
and weight $0<p<1$ is the probability measure $\nu_{t}^{p}$ that
is the distribution of the random real number $\sum_{n=0}^{\infty}\pm t^{n}$,
where the sign is chosen i.i.d. with probability $p,1-p$. One may
view this as the image of the product measure $(p,1-p)^{\mathbb{N}}$
on $\{+,-\}^{\mathbb{N}}$ under the Lipschitz maps $\varphi_{t}(x)=\sum x_{n}t^{n}$.

The following theorem may be inferred from deep existing results,
but follows easily from our methods:

\begin{thm} \label{thm:BC}The lower Hausdorff dimension $\dim_*\nu_{p}^{t}$ is lower
semi-continuous in $(p,t)$. \end{thm}

\subsection{\label{sub:Local-dynamics-and-continuity-of-dim}Local dynamics and
continuity of dimension}

The proofs of Theorems \ref{thm:projection-of-product-invariant-measures},
\ref{thm:convolutions-of-Gibbs-measures}, \ref{thm:projections-of-self-similar-measures-with-irrational-rotation} consist of two independent parts. The first is a
semicontinuity result for the map $\pi\mapsto\dim_{*}\pi\mu$
when $\mu$ is a measure displaying a certain {}``local dynamics''.
Coupled with the general result that $\dim_*\pi\mu$ has the expected dimension for  almost-every $\pi$,
this provides an open dense set of projections which project to nearly
the expected dimension. The second part of the proof relies on some
invariance of the measures (and hence of the set of good projections) under a sufficiently large group to show that this open set, being invariant, is in fact all of
$\Pi_{d,k}$  or some large part of it. While all related works utilize largeness of this action in one way or another, the continuity result is new and is perhaps the main technical innovation of this paper. We outline these results next.

The ``local dynamics'' which we require of a measure $\mu$ is, briefly, that as one zooms in to
a typical point $x$, all the while re-scaling the measure, one sees a
sequence of measures which display stationary dynamics. More precisely, given a measure $\mu$
and $x\in\supp\mu$, one can form a sequence of cubes $B_{n}$ descending to $x$. For example, given an integer $b$ one can choose the $b$-adic cells containing $x$ (later it will be necessary to work with more general cells; see Section \ref{sec:CP-chains-and-local-dynamics}).
Form the sequence of measures $\mu^{x,n}$ obtained by restricting
$\mu$ to $B_{n}$, normalizing, and re-scaling it back to the unit
cube. The sequence $\mu^{x,n}$ is sometimes called the \emph{scenery} at $x$. Our
assumption about $\mu$ will be that for $\mu$-typical $x$, the scenery sequence displays statistical
regularity, i.e. is generic for some distribution $P$ on measures,
and $P$ is independent of $x$.

The limiting distribution $P$ above is the distribution of a so-called CP-chain (or a slight generalization of one), which were introduced by Furstenberg in \cite{Furstenberg70,Furstenberg08} for the purpose of studying some related problems. A CP-chain is a measure-valued Markov process $(\mu_{n})_{n=1}^{\infty}$, in which $\mu_{n}$ are
probability measures on the unit cube (or some other fixed compact
set), and,  conditioned on $\mu_1$,  the sequence $\mu_2,\mu_3 \ldots$ is the scenery of $\mu_1$ at a $\mu_1$-typical point. See Section \ref{sec:CP-chains-and-local-dynamics}. Although CP-chains are quite special objects, in fact many measures
that arise in conformal dynamics have CP-chains associated to them
in a natural way, and are often equal to typical measures for CP-chains
after slight distortion.

Similar notions for measures and sets have been studied by many authors, mostly with the
aim of classifying measures and sets by their limiting local behavior
\cite{Graf95,BedfordFisher96,BedfordFisher97,KriegMorters98,MortersPreiss98,Morters98,BedfordFisherUrbanski02}. See \cite{Hochman09} for a systematic discussion
of CP-chains and their relation to other models of {}``fractal''
measures.

Returning to dimension of projections, one might say that the discontinuity of $\dim\pi\mu$ in $\pi$ is a result of the infinitesimal nature of Hausdorff dimension. It is therefore desirable to express, or at least bound, the dimension in terms of a finite-scale quantity. For this purpose a useful quantity to consider is entropy: For a measure $\nu$ on $\mathbb{R}^{k}$ define the $\rho$-scale entropy
of $\nu$ by \[
  H_{\rho}(\nu)=-\int\log\left(\nu(B_{\rho}(t))\right)d\nu(t)\]
This measures how ``spread out'' $\nu$ is, and its behavior as $\rho\to 0$ has been studied as an alternative notion of dimension (this is so-called entropy or information dimension, which again behaves discontinuously under projections). Our key innovation is to observe that, in the presence of local dynamics, the  (mean) behavior of this entropy at a \emph{fixed finite scale} can be used to bound the dimension of projections of the measure.

\begin{thm} \label{thm:into-local-entropy-averages}Let $\mu$ be a measure on $\mathbb{R}^d$, fix  an integer base $b\geq 2$ and let $\pi\in\Pi_{d,k}$. Suppose that for $\mu$-a.e. $x$, \begin{equation}
  \liminf \frac{1}{N}\sum_{n=1}^{N}H_{1/b}(\pi\mu^{x,n}) > \alpha\label{eq:intro-local-entropy-averages}
\end{equation}
where $\mu^{x,n}$ are the scenery of $\mu$ at $x$ along $b$-adic cells. Then \[
  \dim_*\pi\mu > \frac{\alpha}{\log b} - \frac{C_{d,k}}{\log b},
\]
where the constant depends only on $d,k$.
\end{thm}

The use of $b$-adic cells here is somewhat arbitrary, and we can (and will) sometimes use other filtrations.

Observe now that $H_\rho(\pi\nu)$ is (almost) jointly continuous in $\pi\in\Pi_{d,k}$ and $\nu$, where $\nu$ is a probability measures on the unit cube (actually, it is discontinuous at atomic measures, but the error tends to zero as $\rho\to 0$). Hence for $b$ large enough, as $\pi'\to\pi$, if one  replaces $\pi$ with $\pi'$ in \eqref{eq:intro-local-entropy-averages}, then the new averages also exceed $\alpha$ in the limit, and we obtain a lower bound  $\dim\pi'\mu > \alpha$ for $\pi'$ close enough to $\pi$. Also note that for measures displaying good local dynamics, the limit \eqref{eq:intro-local-entropy-averages} converges to the mean value of $H_{1/b}(\pi\nu)$, with $\nu$ distributed according to the limiting CP-chain of $\mu$. Combining Theorem \ref{thm:into-local-entropy-averages} with some additional analysis leads to the following theorem.

\begin{thm} \label{thm:semicontinuity}Let $P$ be the distribution
of an ergodic $d$-dimensional CP-chain. Then for every $k$ there
is a lower semi-continuous function $E:\Pi_{d,k}\rightarrow\mathbb{R}^{+}$
such that:
\begin{enumerate}
\item $E(\pi) = \min(k,\alpha)$ for almost every $\pi\in\Pi_{d,k}$, where $\alpha$ is the $P$-almost sure dimension of measures in the chain (See Lemma \ref{lem:dimension-of-CP-chains}).
\item For a fixed $\pi\in\Pi_{d,k}$,
\[
\dim_*\pi\mu = E(\pi) \quad\textrm{ for } P-\textrm{a.e. } \mu.
\]
\item There is a set $M$ of measures with $P(M)=1$, such that, for $\mu\in M$,
\[
\dim_*\pi\mu \ge E(\pi) \quad \textrm{ for all } \pi\in \Pi_{d,k}.
\]
\end{enumerate}
\end{thm}

In fact, $E(\pi)$ is the limit, as $\rho\to 0$, of the mean values of the $H_\rho(\pi\nu)$, the mean being over $\nu$ under the distribution $P$.

The following corollary is then immediate:
\begin{cor} \label{cor:open-dense-set}
In the setting of Theorem \ref{thm:semicontinuity}, there exists a set $M$ of measures with $P(M)=1$, such that for every $\varepsilon>0$
there is a dense open set $\mathcal{U}_{\varepsilon}\subseteq\Pi_{d,k}$
satisfying \[
\dim_{*}\pi\mu>\min(k,\alpha)-\varepsilon \quad\mbox{ for all }\pi\in \mathcal{U}_{\e}, \mu\in M.\]
\end{cor}

A weaker result for non-ergodic processes is available, see Theorem
\ref{thm:open-dense-set-of-good-directions-non-ergodic-case} below.

Since differentiable maps are locally close to linear ones, and limits of averages along sceneries at $x$ depend only on the local behavior of the measure near $x$. From this we obtain results on non-linear images of measures:

\begin{thm} \label{thm:lower-semicontiunity-in-C1}Let $P$ be the distribution of an ergodic
CP-chain. Fix $\pi\in\Pi_{d,k}$. Then  for $P$-almost every $\mu$, the map $g \mapsto\dim_{*}g\mu$ is lower semi-continuous at $\pi$
in the $C^{1}$ topology.

Furthermore, the modulus of continuity is
uniform in $\mu$: for every $\varepsilon>0$ there is a $\delta>0$
so that for a.e. $\mu$ if $g\in C^{1}(\supp\mu,\mathbb{R}^{k})$
and $\left\Vert g-\pi\right\Vert_{C^1} <\delta$ then $\dim_{*}g\mu\geq\dim_{*}\pi\mu-\varepsilon$.
\end{thm}

\begin{thm} \label{thm:bound-for-smooth-images}Let $E:\Pi_{d,k}\rightarrow\mathbb{R}$
be the function associated to an ergodic CP-chain as in Theorem \ref{thm:semicontinuity}
and let $\mu$ be typical measure for the chain.\footnote{A priori the function $E$ is defined only on $\Pi_{d,k}$. However it can be extended to all linear functions in a straightforward way so that Theorem \ref{thm:semicontinuity} still holds. Alternatively, one can identify $D_x g$ with the projection onto the $k$-plane orthogonal to the kernel of $D_ x g$.} Then for every $C^1$ map $g:\supp\mu\to\mathbb{R}^k$ without singular points,
\[
\dim_{*}g\mu\geq\essinf_{x\in\supp\mu}E(D_{x}g)\]
\end{thm}

In fact, the above theorem as well as other results involving smooth functions require only differentiability, but for simplicity we present the proofs in the $C^1$ case.

A number of natural questions arise in connection with Theorem \ref{thm:semicontinuity} and the almost-everywhere nature of results for CP-chains. In particular, results which, for each $\pi$, hold almost surely do not a priori hold almost surely for all $\pi$. Another issue we do not resolve here is the behavior of the upper Hausdorff dimension. Some of these issues are addressed in \cite{Hochman09}.

\subsection{\label{sub:outline}Outline of the paper}

Section \ref{sec:General-notation-and-conventions} introduces some
general notation.

Section \ref{sec:Dimension-and-entropy} recalls the notions of entropy
and dimension and some of their properties.

In Section \ref{sec:local-entropy-and-dimension} we study measures
on trees and obtain bounds on the image of such a measure under a
tree morphism.

In Section \ref{sec:Lifting-maps-to-tree-morphisms} we develop machinery
for lifting geometric maps between Euclidean spaces to morphisms between
trees.

Section \ref{sec:Bernoulli-convolutions} discusses Bernoulli convolutions
and Theorem \ref{thm:BC}.

In Section \ref{sec:CP-chains-and-local-dynamics} we define (generalized)
CP-chains.

Section \ref{sec:Dimension-of-projections-and-CP-chains} contains
semicontinuity results for images of typical measures for CP-chains.

In Sections \ref{sec:self-similar-measures}, \ref{sec:Measures-invariant-under-xm} and \ref{sec:convolutions-Gibbs-measures} we prove Theorems \ref{thm:projections-of-self-similar-measures-with-irrational-rotation},
\ref{thm:projection-of-product-invariant-measures} and \ref{thm:convolutions-of-Gibbs-measures}, respectively.

\section{\label{sec:General-notation-and-conventions}General notation and
conventions}

$\mathbb{N}=\{1,2,3\ldots\}$. In a metric space, $B_{r}(x)$ denotes
the closed ball of radius $r$ around $x$. We endow $\mathbb{R}^{d}$
with the sup norm $\left\Vert x\right\Vert =\max\{|x_{1}|,\ldots|x_{d}|\}$
and the induced metric.

All measures are Borel probability measures and all sets and functions
that we work with are Borel unless otherwise noted. The family of Borel probability measures on a metric space $X$ (with the Borel $\sigma$-algebra) will be denoted $\mathcal{P}(X)$.

Given a finite measure $\mu$ on some space and a measurable set $A$,
we write\[
\mu_{A}=\frac{1}{\mu(A)}\mu|_{A}.\]
This is the conditional probability of $\mu$ on $A$.

We use the standard ``big O'' notation for asymptotics: $x=O_p (y) $ means that
$x\leq Cy$, where $C$ depends on the parameter $p$. Similarly $x=\Omega_p(y)$ means $y=O_p(x)$, and $x=\Theta_p(y)$ means $x=O_p(y)$ and $y=O_p(x)$.

For the reader's convenience, we summarize our main notation and typographical conventions in the following table.

\begin{tabular}{lll}
\hline
$d$ &  & Dimension of the ambient Euclidean space.\tabularnewline
$k$ &  & Dimension of the range of a projection.\tabularnewline
$\Pi_{d,k}$ &  & Space of projections from $\mathbb{R}^{d}$ to $k$-dimensional subspaces.\tabularnewline
$\pi$ & & Orthogonal projections. \tabularnewline
$f,g,h$ & & Morphisms between trees (Section \ref{sub:Trees-and-tree-morphisms}), differentiable maps.\tabularnewline
$\alpha,\beta,\gamma$ &  & Dimension of fractal sets and measures.\tabularnewline
$\mu,\nu,\eta,\tau$ &  & Probability measures.\tabularnewline
$P,Q$ &  & Probability distributions on large spaces (e.g. spaces of measures).\tabularnewline
$\dim$ &  & Hausdorff dimension of a set, exact dimension of a measure.\tabularnewline
$\dim_{*},\dim^{*}$ &  & Upper/lower Hausdorff dimension of measures (Section \ref{sec:dimensions-of-measures}).\tabularnewline
$\underline{\dim},\overline{\dim}$ &  & Upper/lower pointwise dimension of a measure (Section \ref{sec:dimensions-of-measures} ).\tabularnewline
$\underline{\dim}_{e}$, $\overline{\dim}_{e}$ &  & Upper/lower Entropy dimension of a measure (Section \ref{sub:Entropy-and-entropy-dim}).\tabularnewline
$H_{\rho}(\cdot)$ &  & $\rho$-scale entropy of a measure (Section \ref{sub:Entropy-and-entropy-dim}).\tabularnewline
$H(\mu,\mathcal{P}),H(X)$ &  & Entropy of $\mu$ w.r.t. partition $\mathcal{P}$ {[}of random variable
$X${]} (Section \ref{sub:Entropy-and-entropy-dim}).\tabularnewline
$\Lambda,\Lambda_{X}$ &  & Symbol set of a tree {[}or of the tree $X${]}; index set of an IFS.\tabularnewline
$\mathcal{D}_b,\mathcal{D}_b (x)$ &  & $b$-adic cell [containing $x$] (Section \ref{sub:-padic-cells-and-filtrations}).\tabularnewline
$X,Y$ &  & Tree (Section \ref{sub:Trees-and-tree-morphisms}) or attractors of an IFS (Section \ref{sub:self-similar-sets-and-measures}, \ref{sub:Gibbs-preliminaries}).\tabularnewline
$x,y$ &  & Points in tree or attractors of an IFS.\tabularnewline
$a,b$ &  & Finite words.\tabularnewline
$[a],[b]$ &  & Cylinder sets in a tree.\tabularnewline
$A^{*}$ &  & Re-scaled version of $A$ (Section \ref{sub:Boxes-and-scaled-measures}).\tabularnewline
$T_{A}$  &  & Re-scaling homothety, mapping $A$ to $A^{*}$ (Section \ref{sub:Boxes-and-scaled-measures}).\tabularnewline
$\mu_{A}$ &  & Conditional measure on $A$ (Section \ref{sub:Dynamics-along-filtrations}).\tabularnewline
$\mu^{A}$ &  & Conditional measure on $A$, re-scaled to $A^{*}$ (Section \ref{sub:Dynamics-along-filtrations}).\tabularnewline
$\Delta,\mathcal{E}$ &  & Partition operator and family of boxes (Section \ref{sub:CP-chains}).\tabularnewline
$\mathcal{U}, \mathcal{U}_\e$ & & Subsets (often open) of $\Pi_{d,k}$ or $C^1(\RR^d,\RR^k)$.\tabularnewline
 &  & \tabularnewline
\hline
\end{tabular}

\section{\label{sec:Dimension-and-entropy}Dimension and entropy}

We denote the Hausdorff dimension of a set $A$ by $\dim A$. Falconer's
books \cite{Falconer90}, \cite{Falconer97} are good introductions
to Hausdorff measure and dimension. In this section we collect some
basic facts about dimension and entropy of measures.

\subsection{\label{sec:dimensions-of-measures}Hausdorff and local dimension
of measures}

Let $\mu$ be a Borel measure on a metric space $X$. The \emph{upper and
lower Hausdorff dimensions} of $\mu$ are given by \begin{align*}
\dim^{*}(\mu) & =\inf\{\dim(A)\;:\;\mu(X\backslash A)=0\},\\
\dim_{*}(\mu) & =\inf\{\dim(A)\;:\;\mu(A)>0\}.\end{align*}
 The \emph{upper and lower local dimensions} of $\mu$ at a point
$x$ are given by \begin{align*}
\overline{\dim}(\mu,x) & =\limsup_{r\rightarrow0}\frac{\log\mu(B_{r}(x))}{\log r},\\
\underline{\dim}(\mu,x) & =\liminf_{r\rightarrow0}\frac{\log\mu(B_{r}(x))}{\log r}.\end{align*}
Clearly $\underline{\dim}(\mu,x)\leq\overline{\dim}(\mu,x)$. The local
dimension of $\mu$ at $x$ exists if $\overline{\dim}(\mu,x)=\underline{\dim}(\mu,x)$,
and is equal to their common value. If the local dimension of $\mu$
exists and is constant $\mu$-almost everywhere, $\mu$ is \emph{exact
dimensional}, and the almost sure local dimension is denoted $\dim\mu$.
Whenever we write $\dim\mu$, we are implicitly assuming that
$\mu$ is exact dimensional. Note that $\dim_{*}\mu=\dim^{*}\mu$ does not imply that $\mu$ is exact dimensional.

We record a few basic facts about lower dimension.

\begin{lem} \label{lem:properties-of-lower-dim}Let $\mu$ be a Borel measure
on a metric space.
\begin{enumerate}
\item \label{eq:Hausdorff-dim-in-terms-of-local-dim}
\begin{equation*}
\dim_{*}(\mu)=\sup\{\alpha\;:\;\underline{\dim}(\mu,x)\ge\alpha\text{ for }\mu\text{-a.e.}x\}.\end{equation*}
\item \label{itm:dim-at-least-inf-over-components} If $P$ is a distribution on measures and $\mu=\int\nu dP(\nu)$,
then \[
\dim_{*}\mu\geq\essinf_{\nu\sim P}\dim_{*}\nu\]
\item If $\mu,\nu$ are equivalent measures (i.e. mutually absolutely continuous), then $\dim_{*}\mu=\dim_{*}\nu$.
\end{enumerate}
\end{lem}

\begin{proof} See e.g. \cite[Proposition 10.2]{Falconer97} for the first assertion; the last two are easy consequences of the definition. \end{proof}

\subsection{\label{sub:-padic-cells-and-filtrations}$p$-adic cells and regular
filtrations}

Given an integer base $p\geq2$, we denote by $\mathcal{D}_{p}$ the
partition of $\mathbb{R}^{d}$ into cubes of the form $I_{1}\times\ldots\times I_{d}$
with $I_{i}=[\frac{k}{p},\frac{k+1}{p})$. Note that $\mathcal{D}_{p^{k}}$,
$k=1,2,\ldots$ form a refining sequence of partitions that separates
points. A cube in $\bigcup_{k=1}^{\infty}\mathcal{D}_{p^{k}}$ is
called a $p$-adic cube.

A sequence $\mathcal{F}=(\mathcal{F}_{n})_{n=1}^{\infty}$ of partitions
of a region in $\mathbb{R}^{d}$ is $\rho$-\emph{regular} if, for
some constant $C>1$, every $B\in\mathcal{F}_{n}$ contains a ball
of radius $\rho^{n}/C$ and is contained in a ball of radius $C\cdot\rho^{n}$.
For example, this assumption is satisfied by $\mathcal{F}_{n}=\mathcal{D}_{p^{n}}$
with $\rho=1/p$.

For a partition $\mathcal{D}$ of $E\subseteq\mathbb{R}^{d}$ and
$x\in E$, we write $\mathcal{D}(x)$ for the unique partition element
containing $x$.

The proof of the following can be found in \cite[Theorem B.1]{KaenmakiRajalaSuomala2010}: \begin{lem}
\label{lem:entropy-along-filtration}Let $\mu$ be a measure on $\mathbb{R}^{d}$
and $\mathcal{F}=(\mathcal{F}_{n})_{n=1}^{\infty}$ a $\rho$-regular
filtration. Then for $\mu$-a.e. $x$,\begin{eqnarray*}
\overline{\dim}(\mu,x) & = & \limsup_{n\rightarrow\infty}\frac{\log\mu(\mathcal{F}_{n}(x))}{n\log\rho},\\
\underline{\dim}(\mu,x) & = & \liminf_{n\rightarrow\infty}\frac{\log\mu(\mathcal{F}_{n}(x))}{n\log\rho}.\end{eqnarray*}
 \end{lem}

\subsection{\label{sub:Entropy-and-entropy-dim}Entropy and entropy dimension}

The various notions of dimension aim to quantify the degree to which
a measure is {}``spread out''. One can also quantify this using
entropy. Given a probability measure $\mu$ on a metric space $X$,
the $r$-scale entropy of $\mu$ is \[
H_{r}(\mu)=-\int\log(\mu(B_{r}(x))\, d\mu(x).\]
 The \emph{upper and lower entropy dimensions} of $\mu$ are defined
as \begin{align*}
\overline{\dim}_{e}(\mu) & =\limsup_{r\rightarrow0}\frac{H_{r}(\mu)}{-\log r},\\
\underline{\dim}_{e}(\mu) & =\liminf_{r\rightarrow0}\frac{H_{r}(\mu)}{-\log r}.\end{align*}
 Clearly $\underline{\dim}_{e}\mu\le\overline{\dim}_{e}\mu$.

Entropy dimensions can also be defined in terms of entropies of partitions.
Recall that if $\mu$ is a probability measure and $\mathcal{Q}$
is a finite or countable partition, then \[
H(\mu,\mathcal{Q})=- \sum_{Q\in\mathcal{Q}}\mu(Q)\log\mu(Q),\]
 with the convention that $0\log0=0$. This quantity is called the \emph{Shannon entropy} of the partition $\mathcal{Q}$.  When $X$ is a random variable taking finitely many values, it induces a finite partition of the underlying probability space. We then write $H(X)$ for the Shannon entropy of this partition, with respect to the associated probability measure. For the basic properties of Shannon entropy, and in particular the definition and properties of conditional entropy, see \cite{CoverThomas06}.

Next, we specialize to $\RR^{d}$.
\begin{lem} \label{lem:entropy-dimension} Let $\mu$ be a probability
measure on $\RR^{d}$ and $n\in\mathbb{N}$. Then there exists a constant
$C=C(d)$ such that \[
|H_{n}(\mu)-H(\mu,\mathcal{D}_{n})|\leq C.
\]
 In particular, for any integer $p\geq 2$, \begin{align*}
\overline{\dim}_{e}(\mu) & =\limsup_{k\rightarrow\infty}\frac{H(\mu,\mathcal{D}_{p^{k}})}{k\log p},\\
\underline{\dim}_{e}(\mu) & =\liminf_{k\rightarrow\infty}\frac{H(\mu,\mathcal{D}_{p^{k}})}{k\log p}.\end{align*}
 \end{lem}
\begin{proof} This is proved for $n=2$ in \cite[Lemma 2.3]{PeresSolomyak00};
the general case is exactly analogous. \end{proof}

The following proposition summarizes some of the relations between
the different notions of dimension. Proofs can be found in \cite{FanLauRao02}.

\begin{prop} \label{prop:relation-dim-measure} Let $\mu$ be a measure
on $\RR^{n}$. Then:
\begin{equation}
\dim_{*}(\mu)\le\underline{\dim}_{e}(\mu)\le\overline{\dim}_{e}(\mu).\label{eq:ineq-dimensions}\end{equation}
If $\mu$ is exact dimensional, then \begin{equation}
\dim_{*}(\mu)=\dim^{*}(\mu)=\underline{\dim}_{e}(\mu)=\overline{\dim}_{e}(\mu)=\dim\mu.\label{eq:eq:exact-dimensional}\end{equation}
 \end{prop}

The next two lemmas establish some continuity properties of $H_{r}(\mu)$

\begin{lem} \label{lem:entropy_comparable_radius} Let $\mu$ be
a probability measure on $\mathbb{R}^{d}$, and fix a constant $C>0$.
Then \[
|H_{r}(\mu)-H_{Cr}(\mu)|=O_{C,d}(1)\quad\text{ for all }r>0.\]
 \end{lem}

\begin{proof} See \cite[Lemma 2.3]{PeresSolomyak00}. \end{proof}

\begin{lem} \label{lem:perturbed_entropy} Let $\mu$ be a probability
measure on the unit ball $B_{1}(0)$, and let $\pi\in\Pi_{d,k}$.
Then for any $C^{1}$ function $g:\mathbb{R}^{d}\rightarrow\mathbb{R}^{k}$
such that $\sup_{x}\|D_{x}g-\pi\|_\infty <r$, \[
|H_{r}(\pi\mu)-H_{r}(g\mu)|=O_{d,k}(1).\]
 \end{lem}

\begin{proof} After a translation we can assume $g(0)=0$. Then it
is easy to see that\[
\pi^{-1}(B_{r}(x))\subseteq g^{-1}(B_{O_{d,k}(r)}(x))\]
and similarly with $\pi$ and $g$ exchanged. The lemma now follows
from Lemma \ref{lem:entropy_comparable_radius}. \end{proof}

\subsection{\label{sub:orth-projections-of-measures}Behavior of measures under
orthogonal projections}

The family $\Pi_{d,k}$ of orthogonal projections from $\mathbb{R}^{d}$
to its $k$-dimensional subspaces may be identified with the Grassmanian
of $k$-planes in $\RR^{d}$, with $V\subseteq\mathbb{R}^{d}$ corresponding
to the projection to $V$. This endows $\Pi_{d,k}$ with a smooth
structure, and hence a measure class (there is also a natural measure,
invariant under the action by the orthogonal group, but we do not
have use for it). Different projections have different images, but
it is convenient to identify them all with $\RR^{k}$ via an affine
change of coordinates, which we specify as needed. Such an identification
is harmless since it does not affect any of the notions of dimension that we use.

If $E\subseteq\mathbb{R}^{n}$ is a Borel set, then the well-known projection
theorem of Marstrand-Mattila says that \[
\dim(\pi(E))=\min(k,\dim(E)),\]
 for almost every projection $\pi\in\Pi_{d,k}$. Hunt and Kaloshin
\cite{HuntKaloshin97} established the analogous results for dimensions
of a measure \cite[Theorem 4.1]{HuntKaloshin97}. In particular we
shall use the following:

\begin{thm} \label{thm:hunt-kaloshin} Let $\mu$ be an exact-dimensional
measure on $\RR^{d}$. Then for almost every projection $\pi\in\Pi_{d,k}$
the projection $\pi\mu$ is exact dimensional, and \[
\dim(\pi\mu)=\min(k,\dim\mu).\]
 \end{thm}

\begin{cor} \label{cor:hunt-kaloshin-entropy}Let $\mu$ be an exact-dimensional
probability measure on $\RR^{d}$. Then for almost every $\pi\in\Pi_{d,k}$,
\begin{equation}
\lim_{r\rightarrow0}-\frac{H_{r}(\pi\mu)}{\log r}=\min(k,\dim\mu)\label{eq:marstrand-for-entropy-dimension}\end{equation}
\end{cor}

\begin{proof} This follows from Proposition \ref{prop:relation-dim-measure}
and Theorem \ref{thm:hunt-kaloshin}. \end{proof}

If $\mu$ is a measure on $\mathbb{R}^d$, then  for every $\pi\in\Pi_{d,k}$ and $x\in\supp(\pi\mu)$ we have  $\overline{\dim}(\pi\mu,\pi x)\leq \overline{\dim}(\pi,x)$. In particular, if $\dim\mu=\alpha$, then $\overline{\dim}(\pi\mu,y)\le \alpha$ for $(\pi\mu)$-a.e. $y$. Also, $\overline{\dim}(\mu,x)\leq d$ almost everywhere.

Finally, combining the above with Lemma \ref{lem:properties-of-lower-dim}\eqref{eq:Hausdorff-dim-in-terms-of-local-dim}, we have:

\begin{cor}\label{cor:proving-expected-dim}
If $\mu$ is an exact-dimensional measure on $\mathbb{R}^d$ and $f:\mathbb{R}^d\to\mathbb{R}^k$ is a differentiable map, then in order to prove that \[
  \dim f\mu = \min\{k,\dim\mu\},
\]
it suffices to prove that \[
  \dim_* f\mu \geq \min\{k,\dim\mu\}.
\]
\end{cor}

\section{\label{sec:local-entropy-and-dimension}Trees, local entropy and
dimension }

In this section we use entropy and martingale methods to bound from
below the dimension of the image $f\mu$ of a measure $\mu$ under
a map $f$. While simple, the result appears to be new; it may be
viewed as a variant of the relative Shannon-McMillan-Breiman with stationarity
replaced by an assumption of convergence of certain averages, and with very few requirements of the ``factor'' map. It may also be viewed as a relative version of \cite[Theorem 2.1]{Furstenberg08}.
Here we formulate the result for morphisms between trees. In the
next section we discuss how to lift more general maps to such morphisms.

\subsection{\label{sub:Trees-and-tree-morphisms}Trees, tree morphisms and metric trees}

A tree is a closed subset $X\subseteq\Lambda^{\mathbb{N}}$. Here
$\Lambda$ is a finite set called the \emph{alphabet} or the \emph{symbol
set}. We usually do not specify the symbol set of a tree and write
it generically as $\Lambda$; if we wish to be specific we write $\Lambda_{X},\Lambda_{Y}$,
etc.

A tree is a compact metrizable totally disconnected space. A basis
of closed and open sets of $\Lambda^{\mathbb{N}}$ (and, in the relative
topology, for $X$) is provided by the cylinder sets. An $n$-cylinder
is a set of the form\[
[a]=\{x\in X\,:\, x_{1}\ldots x_{n}=a\}\]
for $a\in\Lambda^{n}$, and an $n$-cylinder in a tree $X$ is the
intersection of $X$ with such a set.

The cylinders form a countable
family which is partially ordered by inclusion, and this order structure
determines $X$ up to isomorphism (see below for the definition of
a tree morphism). Conversely, any countable partially ordered set
satisfying the obvious axioms gives rise to a tree, and we shall sometimes
represent a tree in this way.

There is a closely related representation of trees as sets of words.
Write $\Lambda^{*}=\bigcup_{n=0}^{\infty}\Lambda^{n}$ (we write $\varnothing$
for the empty word). A tree $X\subseteq\Lambda^{*}$ is characterized
by the set\[
\{a\in\Lambda^{*}\,:\,[a]\cap X\neq\varnothing\}.\]
 In this representation we call the sequence $a\in\Lambda^{n}$ a
node of the tree. Its length is $n$ (denoted $|a|$). If $b\in\Lambda$
and $ab$ is in the tree then $ab$ is the child of $a$ and $a$
is the parent of $ab$. We shall sometimes use the notation\[
a_{1}^{k}=a_{1}a_{2}\ldots a_{k}\]
 to represent the initial $k$-segment of a longer word $a$.

\begin{defn} \label{def:tree-morphism}If
$X$,$Y$ are trees then a morphism is a map $f:X\rightarrow Y$ that
maps $n$-cylinders into $n$-cylinders, i.e.  for all $a\in\Lambda_{X}^{n}$
there exists $b\in\Lambda_{Y}^{n}$ such that $f[a]\subseteq[b]$.
\end{defn} In the symbolic representation of trees this corresponds
to a map $g:\Lambda_{X}^{*}\rightarrow\Lambda_{Y}^{*}$ satisfying
$g(a_{1}\ldots a_{n})=g(a_{1}\ldots a_{n-1})b$ for some $b\in\Lambda_{Y}$.

For $0<\rho<1$, a $\rho$-tree is a tree $X$ together with the compatible
metric \[
d_{\rho}(x,y)=\rho^{\min\{n\,:\, x_{n}\neq y_{n}\}}.\]
 Note that if $X$ is a $\rho$-tree and $Y$ is a $\tau$-tree then
a tree morphism $X\rightarrow Y$ is $\log \tau/\log \rho$-H\"{o}lder. In particular
if $\tau=\rho$ then it is Lipschitz.

In a $\rho$-tree the diameter of an $n$-cylinder is $\rho^{n}$,
and \[
B_{r}(x)=[x_{1}\ldots x_{k}],\]
 where $k=\left\lceil \frac{\log r}{\log\rho}\right\rceil $. Thus
if $\mu$ is a measure on $X$ then\[
\liminf_{n\rightarrow\infty}\frac{\log\mu(B_{r}(x))}{\log r}=\liminf_{n\rightarrow\infty}\frac{\log\mu([x_{1}\ldots x_{n}])}{n\log\rho},\]
and likewise for $\limsup$. In particular, \ref{lem:properties-of-lower-dim}\eqref{eq:Hausdorff-dim-in-terms-of-local-dim} yields

\begin{lem} \label{lem:decay-and-dimension-in-rho-trees}If $\mu$
is a measure on a $\rho$-tree $X$ such that\[
\liminf_{n\rightarrow\infty}\left(-\frac{\log\mu([x_{1}\ldots x_{n}])}{n}\right)\geq\alpha \quad \mu\textrm{-a.e.},
\]
then $\dim_{*}\mu\geq\frac{\alpha}{\log\rho}$. \end{lem}

The metric we choose for a tree is somewhat arbitrary, and we may
change it at our convenience, but one must note that this leads to
a re-scaling of dimensions.

\subsection{\label{sub:Local-entropy-averages-and-dim}Local entropy averages
and mass decay}

Let $\mu$ be a Borel probability measure on a tree $X$. For $x\in X$
denote by $\mu(\cdot|x_{1}^{n})$ the conditional measure on the symbol
set $\Lambda$ given by \[
\mu(x_{n+1}|x_{1}^{n})=\frac{\mu[x_{1}^{n+1}]}{\mu[x_{1}^{n}]}.\]
This is defined only when $\mu[x_{1}^{n}]>0$. The $n$-th information
function is \[
I_{n}(x)=-\log\mu(x_{n}|x_{1}^{n-1}).\]
Thus for $x\in X$,\begin{equation}
-\log\mu[x_{1}\ldots x_{n}]=\sum_{k=1}^{n}I_{k}(x).\label{eq:product-fomula-for-cylinder-measures}\end{equation}

Let $X_{n}$ be the random variable given by projection from $X$
to the $n$-th coordinate and $\mathcal{F}_{n}$ the
$\sigma$-algebra generated by the $X_1,X_2,\ldots,X_n$, i.e. by the $n$-cylinders. Then \[
H(X_{n+1}|x_{1}^{n})=\mathbb{E}(I_{n+1}|\mathcal{F}_{n})(x).\]

\begin{lem} \label{lem:local-entropy-lemma}For $\mu$-a.e. $x$,
\[
-\frac{1}{N}\log\mu[x_{1}^{n}]-\frac{1}{N}\sum_{n=1}^{N}H(X_{n}|x_{1}^{n-1})\rightarrow0.\]
 \end{lem} \begin{proof} Using \eqref{eq:product-fomula-for-cylinder-measures}
we can write this expression as \[
\frac{1}{N}\sum_{n=1}^{N}\left(I_{n}-\mathbb{E}(I_{n}|\mathcal{F}_{n-1})\right)(x).\]
 This is an average of uniformly $L^{2}$-bounded martingale differences,
so by the Law of Large Numbers for martingale differences (see e.g. \cite[Theorem 3 in Chapter VII.9]{Feller71}),
it converges to $0$ a.e. \end{proof}

A related result, which recovers the lemma in the case of doubling measures, can be found in the paper of Llorente and Nicolau \cite[Corollary 3.8 and the discussion following Equation 0.10]{LlorenteNicolau2004}.

Now suppose $f:X\rightarrow Y$ is a tree morphism,
and denote by $f$ also the induced symbolic map $\Lambda_{X}^{*}\rightarrow\Lambda_{Y}^{*}$.
Let $\nu=f\mu$, define the conditional measures $\nu(\cdot|y_{1}^{n})$,
$y_{i}\in\Lambda_{Y}$, as above, and let $Y_{n}$ denote the coordinate
functions on $Y$. Since $f$ is a morphism, we also have the conditional
measures $\nu(\cdot|x_{1}^{n})$ on $\Lambda_{Y}$ given by\[
\nu(b|x_{1}^{n-1})=\sum_{a\in\Lambda_{X}\,:\, f(x_{1}^{n-1}a)=f(x_{1}^{n-1})b}\mu(a|x_{1}^{n-1}),\]
 and we have the corresponding entropy\[
H(Y_{n}|x_{1}^{n-1}).\]
 Note that $\nu(\cdot|x_{1}^{n-1})$ can also be thought of as the
push-forward $f\mu(\cdot|x_{1}^{n-1})$, which is well defined as a measure on $\Lambda_Y$ because
$f$ is a tree morphism.

In the case of a $\rho$-tree $X$ with a probability measure $\mu$ we see, writing again $X_{n}$ for the coordinate functions,
that
\[
H_{\rho}(\mu)=H(X_{1}),
\]
and more generally,
\[
H_{\rho^{n+1}}(\mu_{[x_{1}\ldots x_{n}]})=H(X_{n+1}|x_{1}\ldots x_{n}),
\]
where $\mu_{A}=\frac{1}{\mu(A)}\mu|_{A}$ as usual. If $f:X\rightarrow Y$ is a morphism of $\rho$-trees and $\nu=f\mu$, we have
\[
H_{\rho^{n+1}}(f\mu_{[x_{1}\ldots x_{n}]})=H(Y_{n+1}|x_{1}\ldots x_{n}).
\]

\begin{thm} \label{thm:dimension-via-local-entropy-for-rho-trees}Let
$X$,$Y$ be $\rho$-trees and $f:X\rightarrow Y$ a tree morphism, and
let $\mu$ be a probability measure on $X$. If
\begin{equation} \label{eq:image-entropy-averages}
\liminf\frac{1}{N}\sum_{n=1}^{N}H_{\rho^{n+1}}(f\mu_{[x_{1}\ldots x_{n}]})\geq\alpha
\end{equation}
 for $\mu$-a.e. $x$, then \[
\dim_{*}f\mu\geq\frac{\alpha}{\log(1/\rho)}.\]
\end{thm}
\begin{proof}
We are going to construct a random measure $\nu$ on the tree $X$ satisfying the following properties:
\begin{enumerate}
\item \label{it:same-entropies} Conditioned on $\nu[x_{1}\ldots x_{n}]>0$,
\[
H_{\rho^{n+1}}(\nu_{[x_{1}\ldots x_{n}]}) = H_{\rho^{n+1}}(f\mu_{[x_{1}\ldots x_{n}]})
\]
\item \label{it:isometry-on-support} The restriction of $f$ to $\supp(\nu)$ is injective (and therefore an isometry).
\item \label{it:intensity-is-mu} $\mathbb{E}(\nu(A)) = \mu(A)$ for all Borel sets $A$.
\end{enumerate}
Before proceeding to the construction, we explain how to conclude the proof from these properties. Let $G$ be the (Borel) set where \eqref{eq:image-entropy-averages} holds. By assumption $G$ has full measure, so \eqref{it:intensity-is-mu} tells us that $\nu(G)$ has full measure, almost surely. In view of \eqref{it:same-entropies}, Theorem \ref{thm:dimension-via-local-entropy-for-rho-trees} yields $\dim_* \nu\ge \alpha/\log(1/\rho)$ almost surely. But, by \eqref{it:isometry-on-support}, $f|_{\supp(\nu)}$ is an isometry, and in particular it preserves the dimension of any measure, so we deduce that $\dim_* f\nu\ge \alpha/\log(1/\rho)$ almost surely. Since, using again \eqref{it:intensity-is-mu}, $f\mu=\mathbb{E}(f\nu)$, the desired lower bound on $\dim_*f\mu$ follows from Lemma \ref{lem:properties-of-lower-dim}\eqref{itm:dim-at-least-inf-over-components}.

We now describe the random construction. For a node $a$ on $X$, and a symbol $j\in\Lambda_Y$, we write $F_{a,j}$ for the fiber $\{i\in\Lambda_X: f(ai)=f(a)j\}$. For all $a,j$ such that $F_{a,j}$ is nonempty, we let $\mathbb{P}_{a,j}$ denote the probability distribution on $F_{a,j}$ given by
\[
\mathbb{P}_{a,j}(i)= \frac{\mu(i|a)}{\mu(F_{a,j}|a)}.
\]
To define $\nu$, it is enough to specify the conditional measures $\nu(\cdot|a)$ at each node $a$ in $X$. We proceed by selecting one element from each nonempty fiber $F_{a,j}$, according to $\mathbb{P}_{a,j}$, and passing to it all the mass of the fiber (so that $\nu(i|a)=\mu(F_{a,j}|a)$ whenever $i$ was the chosen element in $F_{a,j}$, and $0$ otherwise), with all the selections independent.

Properties \eqref{it:same-entropies} and \eqref{it:isometry-on-support} are immediate from the construction. It is enough to verify \eqref{it:intensity-is-mu} when $A$ is a cylinder set $[a_1\ldots a_N]$. In this case, writing $b_{n+1}$ for the element of $\Lambda_Y$ such that $f(a_1\ldots a_{n+1})=f(a_1\ldots a_n)b_{n+1}$,

\begin{align*}
\mathbb{E}(\nu[a_1\ldots a_N]) &= \mathbb{P}(\nu[a_1\ldots a_N]>0)\cdot \mathbb{E}(\nu[a_1\ldots a_N]|\nu[a_1\ldots a_N]>0)\\
&=   \prod_{n=0}^{N-1} \frac{\mu(a_{n+1}|a_1\ldots a_n)}{\mu(F_{(a_1,\ldots,a_n),b_{n+1}}|a_1\ldots a_n)}\cdot \prod_{n=0}^{N-1} \mu(F_{(a_1,\ldots,a_n),b_{n+1}}|a_1\ldots a_n)\\
&= \mu[a_1\ldots a_N],
\end{align*}
as desired.
\end{proof}

The significance of this theorem is that one obtains a lower bound
on the dimension of the \textit{image measure} $f\mu$, in terms of
an asymptotic property of the measure $\mu$ \textit{in the domain}.
Because the map is not one-to-one, a lot of structure is destroyed
in the passage from $\mu$ to $f\mu$, and it may be impossible to
find enough structure in $f\mu$ to analyze it directly. Instead,
this theorem allows one to use structural information about $\mu$
and the way $f$ acts on cylinder sets in order to bound $\dim_{*}f\mu$
from below.

\section{\label{sec:Lifting-maps-to-tree-morphisms}Lifting maps to tree morphisms}

In order to make use of the last section's results it is necessary
to lift topological maps (between trees, Euclidean domains, or a mixture
of the two) to tree morphisms. This technical section provides
the tools for this.

\subsection{\label{sub:Base-p-representation}Base-$p$ representation}

Given $p\geq2$ one can represent $[0,1]$ using a $p$-regular tree:
$u:\{0,\ldots,p-1\}^{\mathbb{N}}\rightarrow[0,1]$ is given by\[
u(x)=\sum_{n=1}^{\infty}p^{-n}x_{n}.\]
 We shall always give the full tree $\{0,\ldots,p-1\}^{\mathbb{N}}$
the metric $d_{1/p}$, under which $u$ becomes $1$-Lipschitz. Similarly
the base-$p$ representation of the cube $[0,1]^{d}$ is given by
the tree $(\{0,\ldots,p-1\}^{d})^{\mathbb{N}}$ with the map $(x^{1},\ldots,x^{d})\mapsto(u(x^{1}),\ldots,u(x^{d}))$
and the metric $d_{1/p}$, under which this map is again $1$-Lipschitz
with respect to the $\left\Vert \cdot\right\Vert _{\infty}$-norm
on the range.

\subsection{\label{sub:Faithful-maps}Faithful maps}

Below we introduce a class of maps which do not distort dimension
very much. \begin{defn} \label{def:faithful-maps}Let $X$ be a $\rho$-tree.
A map $f:X\rightarrow\mathbb{R}^{d}$ is \emph{$C$-faithful} if
for each $n$ and each $a\in\Lambda^{n}$ the following conditions hold:
\begin{enumerate}
\item Multiplicity: \label{enum:multiplicity}No point in $f[a]$ is covered
by more than $C$ sets $f[a b],b\in\Lambda$.
\item Decay: \label{enum:decay} $f[a]$ contains a ball of radius $(C^{-1}\rho)^{n}$
and is contained in a ball of radius $(C\rho)^{n}$.
\end{enumerate}
\end{defn} For example, the base-$p$ coding of $[0,1]^{d}$ is $C$-faithful
for $C=2^{d}$.

The second condition in the definition implies that a $C$-faithful
map of a $\rho$-tree is $(1-\frac{\log C}{\log(1/\rho)})$-H\"{o}lder.
Therefore, if $\mu$ is a measure on $X$, then \[
\dim_{*}(f\mu)\leq\frac{\log(1/\rho)}{\log(1/\rho)-\log C}\,\dim_{*}\mu.\]
 In applications, $C$ will be independent of $\rho$ and we will be free to choose $\rho$ to be very small
(for example, representing points in $[0,1]^{k}$ using a large base); then the
bound above approaches $\dim_{*}\mu$. Similarly, the following provides
a lower bound:

\begin{prop} \label{prop:dim-tree-to-euclidean}Let $\mu$ be a measure
on a $\rho$-tree $X$ and suppose $f:X\rightarrow\mathbb{R}^{d}$
is $C$-faithful. Then \[
\dim_{*}(f\mu)\geq\dim_{*}(\mu)-\frac{O_{C,d}(1)}{\log(1/\rho)}.\]
 \end{prop}

\begin{proof} We use the characterization of $\dim_{*}$ given in Lemma \ref{lem:properties-of-lower-dim}\eqref{eq:Hausdorff-dim-in-terms-of-local-dim}.
Given $\varepsilon>0$, Egorov's theorem yields a closed set $E_{\varepsilon}\subseteq X$
of measure at least $1-\varepsilon$ such that \[
\liminf_{r\downarrow0}\frac{\log\mu|_{E_{\varepsilon}}(B_{r}(x))}{\log r}\ge\dim_{*}(\mu)\quad\text{uniformly in }x\in E_\varepsilon.\]
 Since it is enough to prove the desired result for $f(\mu|_{E_{\varepsilon}})$
for each $\varepsilon>0$, we can assume without loss of generality
that there is uniformity already for $\mu$ and, in particular, there
is $N\in\mathbb{N}$ such that \begin{equation}
\mu[a]\le\rho^{(\dim_{*}(\mu)-1)n}\label{eq:uniform-measure-bound}\end{equation}
 whenever $a\in\Lambda^{n}$ with $n\ge N$.

Fix $x\in X$ and write $y=f(x)$. Pick $n\ge N$, and let \[
\Phi=\{a\in\Lambda^{n}\,:\, f[a]\cap B_{\rho^{n}}(y)\neq\varnothing\}.\]
 By the decay hypothesis, each $f[a],a\in\Phi$, contains a ball of
radius $(C^{-1}\rho)^{n}$ and is contained in a ball of radius $(C\rho)^{n}$. In particular, $f[a]\subset B_{(1+2C^n)\rho^n}(y)$.
On the other hand, by the multiplicity assumption, no point can be
covered by more than $C^{n}$ of the sets $f[a],a\in\Phi$. Hence,
writing $\lambda$ for Lebesgue measure on $\mathbb{R}^{d}$, we have
\[
\frac{1}{C^n}(C^{-1}\rho)^{nd}|\Phi|\le\frac{\lambda\left(\bigcup_{a\in\Phi}f[a]\right)}{\lambda(B_{1}(0))} \le \left((1+2C^n)\rho^n\right)^d.
\]
 Therefore $|\Phi|\le \exp(O_{C,d}(n))$, and using \eqref{eq:uniform-measure-bound}
we deduce that \[
(f\mu)(B_{\rho^{n}}(y))\le|\Phi|\max_{a\in\Phi}\mu[a]\le\rho^{n\cdot(\dim_{*}(\mu)-\frac{O_{C,d}(1)}{\log(1/\rho)})}.\]
 Letting $n\rightarrow\infty$ we conclude that \[
\underline{\dim}(f\mu,y)\ge\dim_{*}(\mu)-\frac{O_{C,d}(1)}{\log(1/\rho)}.\]
 In light of Lemma \ref{lem:properties-of-lower-dim}\eqref{eq:Hausdorff-dim-in-terms-of-local-dim}, this
yields the desired result. \end{proof}

Note that if $X$ is a $\rho$-tree then for $a\in\Lambda^{n}$, \[
H_{\rho^{n+1}}(\mu_{[a]})=-\sum_{b\in\Lambda}\mu_{[a]}[ab]\log\mu_{[a]}[ab],\]
i.e. this is the Shannon entropy of $\mu_{[a]}$ with respect to the partition
induced from $a$'s children.

The importance of the following estimate is that it is independent
of $\rho$.

\begin{prop} \label{prop:local-entropy-bound-for-faithful-maps}Let
$\mu$ be a measure on a $\rho$-tree $X$ and suppose $f:X\rightarrow\mathbb{R}^{k}$
is $C$-faithful. Then for any $n$-cylinder $a$, \[
\left|H_{\rho^{n+1}}(\mu_{[a]})-H_{\rho^{n+1}}(f\mu_{[a]})\right|<O_{C,k}(1).\]
 \end{prop}

\begin{proof} Since $f|_{[a]}$ is a $C$-faithful map on the $\rho$-tree
$[a]$ (with the re-scaled metric $\rho^{-n}d_{\rho}(\cdot,\cdot)$)
we may without loss of generality assume that $n=0$, $a=$ empty
word, so we must prove \[
\left|H_{\rho}(\mu)-H_{\rho}(f\mu)\right|<O_{C,k}(1).\]
 Notice that \begin{align*}
H_{\rho}(\mu) & =-\int\log(\mu[x_{1}])d\mu(x),\\
H_{\rho}(f\mu) & =-\int\log\mu(B_{\rho}^{f}(x))d\mu(x),\end{align*}
 where \[
B_{\rho}^{f}(x)=f^{-1}(B_{\rho}(f(x)).\]
 By the decay assumption in Definition \ref{def:faithful-maps}, $f[x_{1}]$
is contained in a ball of radius $C\rho$, and therefore $[x_{1}]\subseteq B_{2C\rho}^{f}(x)$,
whence $H_{2C\rho}(f\mu)\le H_\rho(\mu)$. It then follows from Lemma
\ref{lem:entropy_comparable_radius} that
\[
H_{\rho}(f\mu)-H_\rho(\mu)\le O_{C,k}(1).
\]
For the other inequality, let \[
\Lambda(a)=\{b\in\Lambda:\dist(f[a],f[b])<\rho\}.\]
 A volume argument like the one in the proof of Proposition \ref{prop:dim-tree-to-euclidean}
yields that $b\in\Lambda$ belongs to at most $O_{C,k}(1)$ of the
sets $\Lambda(a)$. Since clearly \[
B_{\rho}^{f}(x)\subseteq\bigcup_{a\in\Lambda(x_{1})}[a],\]
 we can estimate \begin{align*}
H_\rho(\mu)-H_{\rho}(f\mu) & =\int\log\left(\frac{\mu B_{\rho}^{f}(x)}{\mu[x_{1}]}\right)d\mu(x)\\
 & \le\int\frac{\mu B_{\rho}^{f}(x)}{\mu[x_{1}]}d\mu(x)\\
 & \le\int\frac{\sum_{b\in\Lambda(x_{1})}\mu[b]}{\mu[x_{1}]}d\mu(x)\\
 & =\sum_{a\in\Lambda}\sum_{b\in\Lambda(a)}\mu[b]\\
 & \le O_{C,k}(1).\qedhere\end{align*}

\end{proof}

\subsection{\label{sub:Lifting-maps-to-tree-morphisms}Lifting maps to tree morphisms}

The following technical result decomposes a map into a tree morphism,
which is easier to analyze, and a faithful map, which by Proposition \ref{prop:dim-tree-to-euclidean} has little effect
on dimension.  \begin{thm} \label{thm:lifting-maps-to-tre-morphisms}Let
$X$ be a $\rho$-tree and $f:X\rightarrow\mathbb{R}^{k}$ an $L$-Lipschitz
map. Then there is a commutative diagram \begin{diagram}[nohug] X & \rTo^g  &Y\\ & \rdTo<f  & \dTo>h \\ & & \mathbb{R}^k \end{diagram}where:
\begin{enumerate}
\item $Y$ is a $\rho$-tree.
\item $g$ is a tree morphism.
\item $h$ is $O_{k, L}(1)$-faithful.
\item If $\mu$ is a measure on $X$, then for any $n$-cylinder $[a]\subseteq X$,
\[
\left|H_{\rho^{n+1}}(f\mu_{[a]})-H_{\rho^{n+1}}(g\mu_{[a]})\right|=O_{k, L}(1).\]

\end{enumerate}
\end{thm} \begin{proof} We first note that by rescaling the metric on the range $\RR^k$, we may assume that $f$ is $1$-Lipschitz. This rescaling affects the implicit constants in the $O$ notation, but as we allow them to depend on $L$, we obtain an equivalent statement.

The construction of $Y$ and the associated
maps consists of two parts.

\textbf{Step 1: construction of $Y$ and $h$}. Since $f$ is 1-Lipschitz
and $\diam X=1$, we may assume that the image is contained in $[0,1]^{k}$.
Note that the cube $Q=[0,1]^{k}$ has the property that, given
$N$, it can be covered by $2N^{k}$ closed cubes $Q_{0},\ldots,Q_{2^{k}N^{k}-1}\subseteq Q$
such that
\begin{itemize}
\item Each $Q_{i}$ has side length $1/N$,
\item No point in $Q$ is covered by more than $2^{k}+1$ of the cubes $Q_{i}$,
\item If $Q'\subseteq Q$ is a cube of side length $\leq\frac{1}{2N}$,
then $Q'\subseteq Q_{i}$ for some $i$.
\end{itemize}
The same holds for any other cube, with side lengths scaled appropriately.

For example, for $k=1$ cover $[0,1]$ by $2N$ intervals of length
$1/N$ starting at the rational points $i/2N$.

Let $\{N_{n}\}$ be an integer sequence taking values in $\{\lfloor\rho^{-1}\rfloor,\lfloor\rho^{-1}\rfloor+1\}$,
such that for $P_{n}=\prod_{i=1}^{n}N_{i}$ we have $\frac{1}{2}\le\rho^{n}P_{n}\le1$
(if $1/\rho\in\mathbb{N}$ we can take $N_{n}=1/\rho$ and then $P_{n}=\rho^{-n}$).

Let $Y$ be the tree such that each vertex of level $n$ has $2^{k}N_{n}^{k}$
offspring, numbered $0,\ldots,2^{k}N_{n}^{k}-1$ (when $1/\rho\in\mathbb{N}$
this is a regular tree).

We inductively construct a map $\widetilde{h}$ which assigns to each
cylinder set $[a]\subseteq Y$ a cube $\widetilde{h}[a]\subseteq[0,2]^{k}$
of side length $2P_{n}^{-1}$. We start by $\widetilde{h}[\varnothing]=[0,2]^{k}$.
Suppose that $Q=\widetilde{h}[y_{1}\ldots y_{n}]$ has been defined and is a cube
of side length $2P_{n}^{-1}$. Let $Q_{0}\ldots Q_{2N_{n+1}^{k}-1}\subseteq Q$
be the sub-cubes of $Q$ with properties analogous to those listed
above for the unit cube. For $y\in\{0,\ldots,2N_{n+1}^{k}-1\}$ set
$\widetilde{h}([y_{1}\ldots y_{n}y])=Q_{y}$.

Finally, this defines $h$ by \[
\{h(y)\}=\bigcap_{n=1}^{\infty}\widetilde{h}[y_{1}\ldots y_{n}].\]
 It is easy to see that, since $h$ respects inclusion for cylinder
sets, the intersection of the right-hand side is a single point. From
the construction it is easy to check that $h$ is $O_{k}(1)$-faithful.

\textbf{Step 2: defining the morphism $g:X\rightarrow Y$}. It is
more convenient to work with the symbolic representation: we define
a morphism $g:\Lambda_{X}^{*}\rightarrow\Lambda_{Y}^{*}$ so that
$f[a]\subseteq\widetilde{h}[g(a)]$. This clearly implies that
$f=hg$.

We proceed by induction on the word length. Start with $g(\varnothing)=\varnothing$
(corresponding to $g(X)\subseteq Y$). Suppose we have defined $g(x_{1}\ldots x_{n})=y_{1}\ldots y_{n}$
and the cube $\tilde{h}[y_{1}\ldots y_{n}]$, which has side length
$2P_{n}^{-1}$, contains $f[x_{1}\ldots x_{n}]$. Since $f$ is $1$-Lipschitz,
for each $a\in\Lambda_{X}$ the set $f[x_{1}\ldots x_{n}a]$ is contained
in a cube of side length $\rho^{n+1}\le P_{n+1}^{-1}$, i.e. $1/2N_{n+1}$
times the side length of $\widetilde{h}[y_{1}\ldots y_{n}]$. Thus
by construction of $\widetilde{h}$ there is at least one $b\in\{0,\ldots,2N_{n+1}^{k}-1\}$
such that the cube $\widetilde{h}[y_{1}\ldots y_{n}b]$ contains $f[x_{1}\ldots x_{n}a]$;
set $g(x_{1}\ldots x_{n}a)=y_{1}\ldots y_{n}b$.

This completes the construction of $Y$ and of $g,h$.

Finally, the entropy estimate is a consequence of the commutativity
of the diagram, the faithfulness of $h$ and Proposition \ref{prop:local-entropy-bound-for-faithful-maps}.
\end{proof}

\section{\label{sec:Bernoulli-convolutions}Semicontinuity of dimension: Bernoulli
convolutions}

As a warm-up we demonstrate in this section how the methods introduced
so far can be used to obtain semi-continuity of the Hausdorff dimension
of Bernoulli convolutions in the parameter space (Theorem \ref{thm:BC}). Recall that for
$0<t,p<1$ the Bernoulli convolution $\nu_{t}^{p}$ is the distribution
of the random real number\[
\sum_{n=0}^{\infty}\pm t^{n}\]
 where the signs are chosen i.i.d. with marginal distribution $(p,1-p)$.
The parameter $t$ is called the contraction ratio.

Bernoulli convolutions have been studied extensively. It is known
that, with $p=\frac{1}{2}$, almost every $t\in[\frac{1}{2},1)$ leads
to a measure which is absolutely continuous with respect to Lebesgue
(in particular, it has dimension 1), with similar results available
for other values of $p$ in a smaller range of $t$. See \cite{PeresSchlagSolomyak00}
and references therein for further background.

Theorem \ref{thm:BC} can be inferred from a combination of existing
results. In \cite{PeresSolomyak00} it is shown that $\dim_{e}(\nu_{t}^{p})$
exists and is given by the supremum over a countable family of continuous
functions of $t$ and $p$, implying that  $\dim_{e}(\nu_{t}^{p})$ is lower
semicontinuous in $t,p$. On the other hand it is a rather deep fact that $\nu_{t}^{p}$
is exact-dimensional; see \cite{FengHu09} for a careful argument.
Combining these with Proposition \ref{prop:relation-dim-measure}, we find that $\dim(\nu_{t}^{p})$
is lower-semicontinuous. Theorem \ref{thm:BC} provides a direct
argument for the semicontinuity of Hausdorff dimension. We note that a simple ergodicity
argument in the coding space shows that $\dim_{*}(\nu_{t}^{p})=\dim^{*}(\nu_{t}^{p})$
for all $t,p$, so the result for lower Hausdorff dimension implies it for the Hausdorff dimension.

\begin{proof}[Proof of Theorem \ref{thm:BC}]
 Fix $(t_{0},p_{0})\in(0,1)\times(0,1)$
and $\varepsilon>0$, and choose $N$ such that \[
\frac{H_{t_{0}^{N}}\left(\nu_{t_{0}}^{p_{0}}\right)}{N\log(1/t_{0})}>\underline{\dim}_{e}(\nu_{t_{0}}^{p_{0}})-\e.\]
 Write $\Lambda=\{-1,1\}^{N}$. Given $u=(u_{0},\ldots,u_{N-1})\in\Lambda$,
let $P_{t}(u)=\sum_{i=0}^{N-1}u_{i}t^{i}$, and define $\pi_{t}:\Lambda^{\NN}\rightarrow\RR$
by \[
\pi_{t}(x)=\sum_{i=0}^{\infty}P_{t}(x_{i})\, t^{iN}.\]
 Let $\mu^{p}$ be the product measure on $\Lambda^{\mathbb{N}}$
whose marginal is
\[
\mu^p([u]) = p^{|\{i: u_i=1\}|} (1-p)^{|\{i:u_i=-1\}|}   \quad\text{for }u\in\Lambda.
\]
 Then $\nu_{t}^{p}=\pi_{t}(\mu^{p})$. It is not hard to check that
$(t,p)\mapsto H_{t^{N}}(\nu_{t}^{p})$ is continuous. Thus there is
a small square $Q=[t_{0}-\delta,t_{0}+\delta]\times[p_{0}-\delta,p_{0}+\delta]$
(with $\delta$ depending only on $\varepsilon$ since $N=N(\varepsilon)$)
such that, for $(t,p)\in Q$, \[
\gamma:=N\log(1/t)\left(\underline{\dim}_{e}(\nu_{t_{0}}^{p_{0}})-2\e\right)\leq H_{t^{N}}(\nu_{t}^{p}).\]
Fix $(t,p)\in Q$, write $\mu=\mu^{p}$ and $\pi=\pi_{t}$, and set
$\rho=t^{N}$. To complete the proof, it is enough to show that
\[
\dim_{*}(\nu_{t}^{p})  \ge\dim_{*}(\nu_{t_{0}}^{p_{0}})-2\e-O(1/N).
\]
The implicit constant in $O(1/N)$ will depend on $t_0$, but since $N$ can be taken arbitrarily large given $(p_0,t_0)$, this is of no consequence.

Note that, thinking of $X=\Lambda^{\mathbb{N}}$ as a $\rho$-tree,
$\pi_{t}$ becomes Lipschitz and we can apply Theorem \ref{thm:lifting-maps-to-tre-morphisms}
to obtain $X\xrightarrow{g}Y\xrightarrow{h}\RR$.

Since $\mu$ is a product measure, we may identify $\mu_{[a]}$
and $\mu$ under the natural identification of $[a]$ with the full
tree. Also, \[
\pi\mu_{[a]}=S\pi\mu,\]
 where $S:\mathbb{R}\rightarrow\mathbb{R}$ is a homothety that scales
by $t^{N|a|}$. Since translations do not change entropies, we conclude
\[
H_{\rho^{|a|+1}}\left(\pi\mu_{[a]}\right)\ge\gamma.\]
 By Theorem \ref{thm:lifting-maps-to-tre-morphisms}, \[
H_{\rho^{|a|+1}}\left(g\mu_{[a]}\right)\ge\gamma-O(1).\]
 This holds uniformly for all $a\in\Lambda^{*}$. Hence, using Theorem
\ref{thm:dimension-via-local-entropy-for-rho-trees}, \[
\dim_{*}(g\mu)\ge \frac{1}{\log(1/\rho)}(\gamma-O(1)).\]
 Finally, by Proposition \ref{prop:dim-tree-to-euclidean}, \begin{align*}
\dim_{*}(\nu_{t}^{p}) & \ge\dim_*g\mu-\frac{O(1)}{N\log(1/t)}\\
 &\ge \frac{\gamma-O(1)}{N\log(1/t)}\\
 & \ge\underline{\dim}_{e}(\nu_{t_{0}}^{p_{0}})-2\e-O(1/N)\\
 & \ge\dim_{*}(\nu_{t_{0}}^{p_{0}})-2\e-O(1/N),
 \end{align*}
 where we used Proposition \ref{prop:relation-dim-measure} in the
last line. This completes the proof. \end{proof}

\section{\label{sec:CP-chains-and-local-dynamics}CP-chains and local
dynamics}

In this section we formalize the notion of local dynamics of a measure
along a filtration. We then introduce a slight generalization of Furstenberg's
CP-chains, which provide a rich supply of measures with good local
dynamics.

We adopt the convention that a \emph{measure} refers to a probability
measure on Euclidean space or a tree. We use the term \emph{distribution}
for probability measures on larger spaces, such as the space of measures,
sequence spaces over measures, etc.

This section uses some basic notions from ergodic theory. A good introductory reference is \cite{Walters82}.

\subsection{\label{sub:Generic-sequences}Generic sequences}

Let $M$ be a compact metric space (which, later on, we usually
do not specify) and denote by $T$ the shift map on $M^{\mathbb{N}}$.
A sequence $\overline{\mu}=(\mu_{n})_{n=1}^{\infty}\in M^{\mathbb{N}}$
is \emph{generic} for $P\in \mathcal{P}(M^\mathbb{N})$ if the sequence of distributions \[
A_{N}(\overline{\mu},T)=\frac{1}{N}\sum_{n=1}^{N}\delta_{T^{n}\overline{\mu}}\]
 converges in the weak-{*} topology, as $N\rightarrow\infty$ (note that these distributions are the uniform measure on the initial $N$ points of the orbit of $\\overline{mu}$ under $T$). Equivalently, for any $f\in C(M^{\mathbb{N}})$,
\[
\frac{1}{N}\sum_{n=0}^{N-1}f(T^n \overline{\mu})\rightarrow\int fdP.
\]
Note that the limit distribution $P$ is $T$-invariant. Similarly,
$\overline{\mu}$ is \emph{strongly generic} for $P\in \mathcal{P}(M^\mathbb{N})$ if for each $q\geq1$, the corresponding average \[
A_{N}(\overline{\mu},T^{q})=\frac{1}{N}\sum_{n=0}^{N-1}\delta_{T^{qn}\overline{\mu}}\]
 converges as $N\rightarrow\infty$ to a distribution $P_{q}$, and $P_1 =P$.

An ergodic $T$-invariant distribution $P$ on $M^{\mathbb{N}}$ decomposes
under $T^{q}$ into $q'$ ergodic components for some $q'|q$, and
these components average to $P$. When $\overline{\mu}$ is generic
for $P$ and the averages $\lim A_{N}(\overline{\mu},T^{q})$ converge
to a distribution $P_{q}$, then $P_{q}$ is invariant under $T^{q}$
so it is a convex combination of these $q'$ components. Thus,
\[
\frac{1}{q}\sum_{i=0}^{q-1}T^{i}P_{q}=P.
\]
 This implies the following lemma which we record for later use: \begin{lem}
\label{lem:maximizing-integral-along-shifted-sequence}Suppose $\overline{\mu}\in\mathcal{M}^{\mathbb{N}}$
is generic for $P\in\mathcal{P}(\mathcal{M}^{\mathbb{N}})$. Let $q>1$
and suppose that $A_{N}(\overline{\mu},T^{q})\rightarrow Q$. Let
$f\in C(\mathcal{M}^{\mathbb{N}})$. Then there is an $i\in\{0,\ldots,q-1\}$
such that\[
\lim_{N\rightarrow\infty}\frac{1}{N}\sum_{n=0}^{N-1} f(T^{i+qn}\overline{\mu})\geq\int fdP\]
 \end{lem} \begin{proof} Immediate from the fact that $P=\frac{1}{q}\sum_{i=0}^{q}T^{i}Q$.
\end{proof}

\subsection{\label{sub:Boxes-and-scaled-measures}Boxes and scaled measures}

A $d$-dimensional box is a product of $d$ intervals of positive
length, each of which may be open, closed or half-open. The eccentricity
of a box is the ratio of the lengths of the longest and shortest side.

A box is normalized if its volume is $1$ and its {}``lower left
corner'', i.e. the lexicographically minimal point in its closure,
is at the origin. For instance $[0,1]^{d}$ and $(0,1)^{d}$ are normalized.
Every box $B$ can be scaled and translated in a unique way to give
a normalized box, which we denote $B^{*}$. We also define the linear
operator $T_{B}$ by \[
T_{B}(x)=\frac{1}{\vol B}(x-\min B),\]
 where $\min$ refers to the lexicographical ordering. Thus $B^{*}=T_{B}B$.

Recall that for a measure $\mu$ and box $B$ with $\mu(B)>0$ we
write \[
\mu_{B}=\frac{1}{\mu(B)}\mu|_{B},\]
 which is a probability measure supported on $B$. We also define \[
\mu^{B}=T_{B}(\mu_{B})=\frac{1}{\mu(B)}T_{B}(\mu|_{B}).\]
 This is a probability measure supported on $B^{*}$. We call this
the re-scaled version of $\mu_{B}$.

\subsection{\label{sub:Dynamics-along-filtrations}Dynamics along filtrations }

Suppose now that we are given a $\rho$-regular sequence of refining
partitions $\mathcal{F}=(\mathcal{F}_{n})_{n=1}^{\infty}$ of a box
$B$ into sub-boxes. For $\mu$ supported on $B$ and $x\in\supp\mu$
we write $\mathcal{F}_{n}(x)$ for the element of $\mathcal{F}_{n}$
containing $x$. Define\begin{eqnarray*}
\mu_{x,n} & = & \mu_{\mathcal{F}_{n}(x)}\end{eqnarray*}
 and\[
\mu^{x,n}=\mu^{\mathcal{F}_{n}(x)}\]
In this way, for each $\mu$ and $x$ we obtain sequences of measures
$(\mu_{x,n})_{n=1}^{\infty}$ and $(\mu^{x,n})_{n=1}^{\infty}$. The
former sequence does not exhibit interesting dynamics, since the support
of the measures decreases to a point. However, elements of the latter sequence
have been re-scaled and the sequence potentially exhibits interesting dynamics.
Because the filtration is $\rho$-regular the eccentricities of $\mathcal{F}_{n}(x)$
are bounded, so there is a bounded region of $\mathbb{R}^{d}$ supporting
all the measures $\mu^{x,n}$. Thus all the $\mu^{x,n}$ belong to some weak-$^*$
compact set in the space of measures.

\begin{defn} \label{def:measures-generating-along-a-partition}Given
a measure $\mu$ on $\mathbb{R}^{d}$ and a $\rho$-regular sequence
$\mathcal{F}=(\mathcal{F}_{n})_{n=1}^{\infty}$ of partitions, $\mu$
\emph{generates} a distribution $P\in\mathcal{P}(M^{\mathbb{N}})$
at $x$ if $(\mu^{x,n})_{n=1}^{\infty}$ is strongly generic for $P$.
It generates $P$ along $\mathcal{F}$ if it generates $P$ at $\mu$-a.e.
point. \end{defn}

We record for later use the following useful result:

\begin{lem} \label{lem:restricting-preserves-generation}Suppose
$\mu$ generates $P$ w.r.t. a partition $(\mathcal{F}_{n})$. Let
$E$ be a set with $\mu(E)>0$. Then $\nu=\mu_{E}$ generates $P$.\end{lem}

\begin{proof} Write $\nu=\mu_{E}$. By the martingale theorem, for
$\mu$-a.e. $x\in E$, \[
\frac{\nu(\mathcal{F}_{n}(x))}{\mu(\mathcal{F}_{n}(x))}=\frac{\mu(E\cap\mathcal{F}_{n}(x))}{\mu(\mathcal{F}_{n}(x))}\rightarrow1\]
 which implies that, for a.e. $x\in E$ the sequences $(\mu^{x,n})_{n=1}^{\infty}$
and $(\nu^{x,n})_{n=1}^{\infty}$ are weak-{*} asymptotic as $n\rightarrow\infty$, so if one is generic for some distribution, both are.
\end{proof}

\subsection{\label{sub:CP-chains}CP-chains}

We next introduce a slightly generalized version of Furstenberg's
CP-chains, which will supply us with measures and filtrations leading
to generic sequences.

\begin{defn} \label{def:partition-operator}Let $\mathcal{E}$ be
a family of boxes. A \emph{partition operator} $\Delta$ on $\mathcal{E}$
assigns to each $E\in\mathcal{E}$ a partition $\Delta E=\{E_{i}\}\subseteq\mathcal{E}$
of $E$ in a translation and scale-invariant manner, i.e. if $Sx=ax+b$
and $E,SE\in\mathcal{E}$ then $S(\Delta E)=\Delta(SE)$. \end{defn}

Define the iterates of $\Delta$ by $B\in\mathcal{E}$ by \begin{eqnarray*}
\Delta^{0}(B) & = & \{B\},\\
\Delta^{n+1}(B) & = & \bigcup_{E\in\Delta^{n}B}\Delta(E).\end{eqnarray*}
 Thus $\Delta^{n}(B)$ form a sequence of refining partitions of $B$.
\begin{defn} \label{def:regular-partition-operators}A partition
operator $\Delta$ on $\mathcal{E}$ is $\rho$-regular if for each
$B\in\mathcal{E}$ the sequence of partitions $(\Delta^{n}B)_{n=1}^{\infty}$
is $\rho$-regular. \end{defn} For example, the base-$b$ partition
operator is defined on $\mathcal{E}=\{[u,v]^{d}\,:\, u<v\}$ by $\Delta([0,1]^{d})=\mathcal{D}_{b}$
(and extend by invariance to all cubes). Then $\Delta^{n}([0,1]^{d})=\mathcal{D}_{b^{n}}$.
This operator is $1/b$-regular. \begin{defn} \label{def:A-CP-chains}A
CP-chain for a $\rho$-regular partition operator $\Delta$ on $\mathcal{E}$
is a stationary Markov process $(\mu_{n},B_{n})_{n=1}^{\infty}$,
where:
\begin{enumerate}
\item The state space is the space of pairs $(\mu,B)$ in which $B\in\mathcal{E}$
is a box and $\mu$ is a probability measure supported on $B^{*}$.
\item The transition is given by the law \[
\mbox{ for }B\in\Delta(A^{*})\;,\;(\mu,A)\mapsto(\mu^{B},B)\;\mbox{ with probability }\mu(B).\]

\end{enumerate}
\end{defn}

We usually do not specify $\Delta$ (or $\mathcal{E}$), and use this
symbol generically for the partition operator associated to a CP-chain.

The stationary process $(\mu_{n})_{n=1}^{\infty}$ is called the \emph{measure
component} of the process. We shall not distinguish notationally between
the distribution of the CP-chain, its measure component and its
marginals. Thus if $P$ is the distribution of a CP-chain we may
write $(\mu,B)\sim P$, $\mu\sim P$ etc; the meaning should be clear
from the context.

Furstenberg's CP-chains are recovered using the base-$b$ partition
operator. We use this partition operator everywhere except in the proof of
Theorem \ref{thm:projection-of-product-invariant-measures}, where
a slightly more elaborate partition operator is needed. We remark
that one can introduce even more general CP-chains by allowing
the partition to depend also on the measure, i.e. $B_{n+1}=\Delta(B_{n},\mu_{n})$,
and also allow more general shapes than boxes; but we shall not need
this.

The following consequence of the ergodic theorem is immediate:

\begin{prop} \label{prop:CP-chain-measures-generate-the-process}Let
$(\mu_{n},B_{n})_{n=1}^{\infty}$ be an ergodic CP-chain with partition
operator $\Delta$ and distribution $P$. Then for $P-$a.e. $(\mu,B)$,
$\mu$ generates the measure component of $P$ w.r.t. the partitions
$\mathcal{F}_{n}=\Delta^{n}(B^{*})$, $n=1,2,3\ldots$.\end{prop}

\begin{proof} Given a typical $(\mu_{1},B_{1})$, consider the conditional distribution on $(\mu_{n},B_{n})_{n=2}^{\infty}$,
obtained by running the chain forward from $(\mu_{1},B_{1})$ using the transition law in
the definition. Note that for $n\geq 2$, the random set \[
A_{n}=T_{B_{2}}^{-1}\ldots  T_{B_{n-1}}^{-1}T_{B_n}^{-1}B_{n}^{*}\]
 satisfies $A_{n}\in\Delta^{n-1}(B_{1}^{*})$ and \[
\mu_{n}=(\mu_{1})^{A_{n}}.\]
 Since $\diam A_{n}\rightarrow0$ by regularity of $\Delta$, the
intersection $\bigcap_{n=1}^{\infty}A_{n}$ consists almost surely
of a single random point $X\in B_{1}^*$. By definition\[
\mu_{n}=(\mu_{1})^{X,n},\]
 and furthermore the transition law has been so chosen that $X$ is
distributed according to $\mu_{1}$.

Now, by the ergodic theorem almost every realization $(\mu_{n},B_{n})_{n=1}^{\infty}$
is generic for the CP-chain, and in particular almost every $(\mu_{n})_{n=1}^{\infty}$
is generic for the measure component of the process. Hence for almost
every $\mu_{1}$ and almost every $(\mu_{1})_{n=1}^{\infty}$ conditioned
on $\mu_{1}$ this is true. But by the above, given $\mu_{1}$ the
conditional distribution on sequences $(\mu_{n})_{n=1}^{\infty}$
of measures is the same as the distribution $((\mu_{1})^{x,n})_{n=1}^{\infty}$
when $x$ is distributed according to $\mu$, as desired.

Strong genericity follows in the same way, because almost every point
in an ergodic system is strongly generic.\end{proof}

\begin{cor} \label{cor:non-ergodic-CP-chains-generate}Let $(\mu_{n},B_{n})_{n=1}^{\infty}$
be a CP-chain with partition operator $\Delta$. Let $P_{(\mu,B)}$
denote the ergodic component of $(\mu,B)$. Then for a.e. pair $(\mu,B)$,
$\mu$ generates the measure component of $P_{(\mu,B)}$, w.r.t. the
filtration $\mathcal{F}_{n}=\Delta^{n}(B^{*})$.\end{cor}

\begin{proof} This follows from the previous proposition and the
fact that the ergodic components of a Markov chain are Markov chains
for the same transition law. \end{proof}

The next lemma is analogous to \cite[Theorem 2.1]{Furstenberg08}
(and the remark following it). In the examples we shall encounter one
can either rely on that proposition, or else the statement will be clear for other
reasons, but we outline a proof for completeness.

\begin{lem} \label{lem:dimension-of-CP-chains}For an ergodic
CP-chain $(\mu_{n},B_{n})_{n=1}^{\infty}$, a.e.
measure $\mu_{n}$ is exact dimensional, and $\dim\mu_{n}$ is almost surely constant.\end{lem}

\begin{proof} It is easy to see that $\dim_{*}\mu_{n}$ is non-increasing in $n$,
since $\mu_{n+1}$ is, up to scale and translation, the restriction
of $\mu_{n}$ to $B_{n+1}$. Since $\dim_{*}(\cdot)$ is a Borel
function of the measure, and the process is ergodic, $\dim_{*}(\cdot)$
must be almost everywhere constant. The same argument works for $\dim^{*}$.

To see that $\mu_1$ (and hence $\mu_n$) is exact-dimensional, condition the process on $(\mu_{1},B_{1})$,
and let $A_{n}$ be defined as in the proof of Proposition \ref{prop:CP-chain-measures-generate-the-process}.
Then\begin{eqnarray*}
\lim_{N\rightarrow\infty}-\frac{\log\mu_{1}(A_{N})}{\log\diam A_{N}} & = & \lim_{N\rightarrow\infty}-\frac{1}{N\log\rho}\sum_{n=1}^{N}\frac{\log\mu_{1}(A_{n})}{\log\mu_{1}(A_{n-1})}\\
 & = & \lim_{N\rightarrow\infty}-\frac{1}{N\log\rho}\sum_{n=0}^{N-1}\log\mu_{n}(B_{n+1}).\end{eqnarray*}
 (in the first equality we used $\rho$-regularity of the partition
operator). Writing $P$ for the distribution of the process, the ergodic
theorem implies that the above converges to \[
\alpha=-\frac{1}{\log\rho}\int H(\mu,\Delta (B^*))\; dP(\mu,B)\]
 almost surely. Using the fact that $X=\cap\overline{A_{n}}$ is distributed
according to $\mu$, and using regularity of $\Delta$ again, it follows from Lemma \ref{lem:entropy-along-filtration} that \[
\overline{\dim}(\mu_{1},x)=\underline{\dim}(\mu_{1},x)=\alpha\]
 for $\mu_{1}$-a.e. $x$, which establishes the lemma. \end{proof}

In light of the previous lemma, we refer to the dimension of a typical measure for a CP-chain as the dimension of the chain.

\subsection{Micromeasures and existence of CP-chains}

The following discussion is adapted from \cite{Furstenberg08}.

Let $\mu$ be a measure on $\mathbb{R}^d$. A \emph{micromeasure} of $\mu$ is any weak limit of measures of the form $\mu^{Q_n}$, where the $Q_n$ are cubes of side length tending to $0$. The set of micromeasures of $\mu$ is denoted $\left\langle \mu\right\rangle$. Micromeasures are closely related to the tangent measures of geometric measure theory.

Starting from a measure $\mu$ on $[0,1]^d$  and the $b$-adic partition operator, one can run the chain forward. If one averages the distributions at times $1,2,\ldots,n$ one gets a sequence of distributions which in general will not converge, but one may still take weak-* limits of it. These limiting distributions can be easily shown to be CP-chains and are supported on the micromeasures of $\mu$. In this way we have associated to $\mu$ a family of CP-chains supported on $\left\langle \mu\right\rangle$, and from which one may hope to extract information about $\mu$. One result in this direction is the following theorem, which appears in another form in Furstenberg's paper \cite{Furstenberg70}. Since it is not stated there in this way, we indicate a proof.

\begin{thm} \label{thm:furstenberg} Let $\mu$ be a measure on $[0,1]^d$ and $b\geq 2$. Then there is an ergodic base-$b$ CP-chain of dimension at least $\overline{\dim}_e(\mu)$ supported on $\left\langle \mu\right\rangle$.
\end{thm}

\begin{proof}

This is very similar to \cite[Proposition 5.2]{Furstenberg08}. For completeness we give a proof outline.
First, fix a base $b$ and choose a sequence $\ell(n)\rightarrow\infty$ such that \[ \limsup_{n\rightarrow\infty}\frac{1}{\ell(n)\log b}H(\mu,\mathcal{D}_{b^{\ell(n)}})\geq\overline{\dim}_{e}(\mu)\]
as we may do by Lemma \ref{lem:entropy-dimension}. Let $\mathcal{P}$ denote the space of
Borel probability measures on $[0,1]^{d}$, equipped with the weak-{*}
topology, which makes it compact and metrizable.

Given $x\in[0,1]$ let $\mu(x,i)=\mu^{\mathcal{D}_{b^{i}}(x)}$ and
$D(x,i)=T_{\mathcal{D}_{b^{i}}(x)}x$. Let $P_{n}$ denote the distribution
on $\mathcal{P}\times\mathcal{D}_{b}$, given by \[
P_{n}=\frac{1}{\ell(n)}\sum_{k=1}^{\ell(n)}\delta_{(\mu(x,k),D(x,k-1))}\]
where $x\in[0,1]^{d}$ is initially chosen with distribution $\mu$.
Let $P$ be any subsequential limit of $P_{n}$ as $n\rightarrow\infty$,
where the topology on $\mathcal{P}\times\mathcal{D}_{b}$ is the product
of the weak-{*} topology and the discrete topology on $\mathcal{D}_{b}$.
Then it is simple to verify that $P$ will be a CP-distribution as long as $P$-a.e. $(\nu,D)\in\mathcal{P}\times\mathcal{D}_{b}$
satisfies $\nu([0,1)^{d})=1$. This may not be the case in general,
since the measures $\mu(x,i)$, as $i\to\infty$, may become increasingly concentrated on the boundary of the cube. However, we can apply the following reduction.
Let $h(x)=ax+c$ denote a random homothety, where $a\in(0,\frac{1}{2})$
and $c\in[0,\frac{1}{2})^{d}$ are chosen uniformly with respect to
Lebesgue. Note that $h\mu$ is still supported on $[0,1]^{d}$, and
one can show \cite[Theorem \ref{NOT-KNOWN-YET}]{Hochman09} that for a.e. choice of $h$, if we replace
$\mu$ by $h\mu$ and carry out the construction above, then $P$-a.e.
$(\nu,D)$ does in fact satisfy $\nu([0,1)^{d})=1$, and hence $P$
is a CP-distribution. Also, $h\mu$ clearly has the same micromeasures as $\mu$. We assume that $\mu$ has been perturbed
in this manner and the condition above is satisfied. Note that the
measure component of $P$ is supported on $\left\langle \mu\right\rangle $

One may now verify that for large enough $n$,
\[
\left| \int \frac{1}{\log b}H(\nu,\mathcal{D}_{b})\, dP_{n}(\nu) - \frac{1}{\ell(n)\log b}H(\mu,\mathcal{D}_{b^{\ell(n)}})\right| \leq \frac{1}{\ell(n)}.
\]
This can be derived by integrating Equation \eqref{eq:product-fomula-for-cylinder-measures} using the tree corresponding to the partitions $\mathcal{D}_{b^{i+j}}$, $i=0,\ldots,N=[\ell(n)/k]-1$, and summing over $j=0,\ldots,b^k-1$. See also \cite[Theorem 2.1]{Furstenberg08}. It follows that\[
\int \frac{1}{\log b}H(\nu,\mathcal{D}_{b})\, dP(\nu)=\overline{\dim}_{e}(\mu)\]
(although entropy is discontinuous this follows from the assumption that $P$ is supported on pairs $(\nu,D)$ for which $\nu$ gives mass 0 to the boundaries of elements of $\mathcal{D}$. This also holds a.s. over the choice of the map $h$ above). Hence there is an ergodic component
$Q$ of $P$ such that \[
\int \frac{1}{\log b}H(\nu,\mathcal{D}_{b})\, dQ(\nu) \geq \overline{\dim}_{e}(\mu)\]
and $Q$ may be taken to be a CP-distribution, because typical ergodic
components of CP-distributions are CP-distributions.

It remains to
show that $\dim\nu\geq\overline{\dim}_{e}(\mu)$ for $Q$-a.e. $\nu$.
This is an application of the entropy averages method and the ergodic theorem, similar to the proof of  Lemma \ref{lem:dimension-of-CP-chains} (see also \cite[Theorem 2.1]{Furstenberg08}. 

\end{proof}

\section{\label{sec:Dimension-of-projections-and-CP-chains}Dimension of
projections and CP-chains}

In this section we establish some continuity results for linear and smooth projections
of typical measures for CP-chains. We show that the dimension of
these projections is controlled, or at least bounded below, by mean
projected entropies of the CP-chain.

\subsection{\label{sub:Linear-projections}Linear projections}

Fix an ergodic CP-chain $(\mu_{n},B_{n})_{n=1}^{\infty}$ with partition
operator $\Delta$ and distribution $P$. Fix $k$ and an orthogonal projection $\pi\in\Pi_{d,k}$. Given a measure $\nu$ and $q\in\mathbb{N}$, let \[
e_{q}(\nu)=\frac{1}{q\log(1/\rho)}H_{\rho^{q}}(\pi\nu),\]
 and denote the mean value of $e_{q}$ by
\[
E_{q}=\int e_{q}(\nu)dP(\nu).
\]

For the rest of the section, the constants implicit in the $O(\cdot)$
notation depend only on $\rho$, the constant in the definition of $\rho$-regularity,
$d$ and $k$. The following theorem contains the proof of Theorem \ref{thm:into-local-entropy-averages}

\begin{thm} \label{thm:linear-projections-of-CP-chains}Let $P$
be the distribution of an ergodic CP-chain, $\pi\in\Pi_{d,k}$ a projection,
and let $e_{q},E_{q}$ be defined as above. Then if $\mu$ is a measure
generating $P$ along a filtration $\mathcal{F}_{n}=\Delta^{n}B^*$,
then \begin{equation}
\dim_{*}(\pi\mu)\ge E_{q}-O(1/q)\label{eq:lower-bound-to-prove}.\end{equation}
In particular, this holds for $P$-a.e. $\mu$.
\end{thm}

\begin{proof}
First, suppose that the measure component of the process
is totally ergodic. Since $\mu$ generates the measure component of
$P$, by Proposition \ref{prop:CP-chain-measures-generate-the-process},
\[
\frac{1}{N}\sum_{n=0}^{N-1}e_{q}(\mu^{x,n})\rightarrow E_{q}.
\]

Fix $q$. Using linearity of $\pi$, the $\rho$-regularity of $\mathcal{F}$
and Proposition \ref{prop:local-entropy-bound-for-faithful-maps} we see
that for every $n\in\mathbb{N}$, \[
\left|H_{\rho^{q}}(\pi\mu^{x,n})-H_{\rho^{n+q}}(\pi\mu_{x,n})\right| = O(1).\]
 Therefore for $\mu$-a.e. $x$,\[
\frac{1}{q\log(1/\rho)}\cdot\liminf_{N\rightarrow\infty}\frac{1}{N-1}\sum_{n=0}^{N}H_{\rho^{n+q}}(\pi\mu_{x,n})\ge E_{q}-O(1/q).
\]

Let $X$ be the $\rho^{q}$-tree whose nodes at level $n$ are the
atoms of $\mathcal{F}_{qn}$ with ancestry determined by inclusion.
Define $f:X\rightarrow \mathbb{R}^d$ by
\[
\{f(A_{1},A_{2},\ldots)\}=\bigcap_{n=1}^{\infty}\overline{A_{n}}.
\]
Then $f$ is Lipschitz by $\rho$-regularity of $\mathcal{F}$ (the Lipschitz
constant depends also on the constant $C$ in the definition of $\rho$-regularity).
Let $\widetilde{\mu}$ be the lift of $\mu$ to $X$.%
\footnote{If $\mu$ gives non-zero mass to the boundaries of partition elements
the lift may not be unique. Fix for example the lift defined
by the condition that the cylinder set corresponding to $E\in\mathcal{F}_{n}$
has mass $\mu(E)$, and we choose $\mu$ to be this measure. Alternatively,
we may assume that the boundary of partition elements is null by reducing,
if necessary, to a lower-dimensional case as in \cite{Furstenberg08}.%
}

For the map $\widetilde{f}=\pi f:X\rightarrow\mathbb{R}^{k}$ apply
Theorem \ref{thm:lifting-maps-to-tre-morphisms}, obtaining a $\rho^{q}$-tree
$Y$ and maps $X\xrightarrow{g}Y$$\xrightarrow{h}\mathbb{R}^{k}$
as in the theorem. It follows that for each $n$-cylinder $\widetilde{E}\subseteq X$,
corresponding to $E\in\mathcal{F}_{qn}$, we have\[
\left|H_{\rho^{q(n+1)}}(\widetilde{f}\widetilde{\mu}_{\widetilde{E}})-H_{\rho^{q(n+1)}}(g\widetilde{\mu}_{\widetilde{E}})\right|=O(1).\]
 Thus, for $g\widetilde{\mu}$-a.e. $y\in Y$, by total ergodicity, for $\widetilde{\mu}$-a.e. $x\in X$,\[
\frac{1}{q\log(1/\rho)}\liminf_{N\rightarrow\infty}\frac{1}{N}\sum_{n=0}^{N-1}H_{\rho^{q(n+1)}}(g\widetilde{\mu}_{[y_{1}\ldots y_{n}]})\geq E_{q}-O(1/q).
\]
 By Theorem \ref{thm:dimension-via-local-entropy-for-rho-trees},
this implies that \[
\dim_{*}g\widetilde{\mu}\geq E_{q}-O(1/q),\]
 and since $h$ is faithful, \[
\dim_{*}hg\widetilde{\mu}\geq E_{q}-O(1/q).\]
 As $hg\widetilde{\mu}=\pi\mu$, we are done.

Suppose now that the measure component of the process is not totally
ergodic. For $\mu$-almost every $x$, there is, by Lemma \ref{lem:maximizing-integral-along-shifted-sequence},
an $i=i(x)\in\{0,\ldots,q-1\}$ (which may be chosen measurably in $x$)
such that \[
\liminf\frac{1}{N}\sum_{n=0}^{N-1}e_{n}(\mu^{x,qn+i})\geq E_{q}.
\]
 Let $A_{i}\subseteq\mathbb{R}^{d}$ be the partition according to
$i(x)$. We may apply the argument above to $T_{B}\mu_{A_{i}\cap B}$
for each $i$ and each $B\in\mathcal{F}_{i}$ separately, using the
induced filtrations $T_{B}\mathcal{F}$ (see Lemma \ref{lem:restricting-preserves-generation}).
Since $\mu$ is a weighted average of the measures $\mu_{A_{i}\cap F}$,
this completes the proof. \end{proof}

We now let $\pi$ vary. Thus $e_{q}:\mathcal{M}\times\Pi_{d,k}\rightarrow[0,k]$
and $E_{q}:\Pi_{d,k}\rightarrow[0,k]$, and we write $e_{q}(\nu,\pi)$
and $E_{q}(\pi)$ to make the dependence on $\pi$ explicit. Note that the next theorem  implies Theorem \ref{thm:semicontinuity}.

\begin{thm} \label{thm:semicontinuity-1}Fix an ergodic CP-chain of dimension $\alpha$ (recall Lemma \ref{lem:dimension-of-CP-chains}) with  distribution $P$. Define
$e_{q}$ and $E_{q}$ as above. The limit \[
E(\pi):=\lim_{q\rightarrow\infty}E_{q}(\pi)\]
 exists and $E:\Pi_{d,k}\rightarrow[0,k]$ is lower semi-continuous. Moreover,
\begin{enumerate}
\item $E(\pi) = \min(k,\alpha)$ for almost every $\pi$. \label{itm:E-is-ae-constant}
\item For a fixed $\pi\in\Pi_{d,k}$, \label{itm:dim-equals-E-ae}
\[
\underline{\dim}_e(\pi\mu) = \dim_*\pi\mu =  E(\pi) \quad\textrm{ for } P-\textrm{a.e. } \mu.
\]
\item For any measure $\mu$ that generates (the measure component of) $P$ along a filtration $\{\Delta^n B^*\}$,
\[
\underline{\dim}_e(\pi\mu) \ge \dim_*\pi\mu \ge E(\pi) \quad \textrm{ for all } \pi\in \Pi_{d,k}.
\]
In particular, the above holds on a set $M$ with $P(M)=1$. \label{itm:itm-dim-at-least-E-for-generating-measure}
\end{enumerate}
\end{thm}

\begin{proof} We first establish convergence of $E_q$ and claim \eqref{itm:dim-equals-E-ae}. It follows from Theorem \ref{thm:linear-projections-of-CP-chains} that if $\mu$ generates $P$, then
we obtain \begin{equation}
\dim_{*}(\pi\mu)\ge\limsup E_{n}(\pi).\label{eq:lower-bound-for-H-dim-of-projection}\end{equation}
 On the other hand, by definition of entropy dimension, we have \[
\underline{\dim}_{e}(\pi\mu)=\liminf_{n\rightarrow\infty}e_{n}(\mu,\pi),\]
 Integrating, we have by Fatou that \begin{equation}
\int\underline{\dim}_{e}(\pi\mu)dP(\mu)\le\liminf_{n}E_{n}(\pi).\label{eq:upper-bound-for-e-dim-of-projection}\end{equation}
Since $\dim_{*}\pi\mu\leq\underline{\dim}_{e}\pi\mu$ holds for any
measure by equation \eqref{eq:ineq-dimensions} in Proposition \ref{prop:relation-dim-measure},
combining \eqref{eq:lower-bound-for-H-dim-of-projection} and \eqref{eq:upper-bound-for-e-dim-of-projection}
we see that $E_{n}(\pi)$ converges and the limit is the common value
of $\dim_{*}\pi\mu=\underline{\dim}_{e}\pi\mu$ for almost every $\mu$ (possibly depending on $\pi$).

To prove \eqref{itm:E-is-ae-constant}, write $\beta = \min\{k,\alpha\}$ for the expected dimension of the image measure. For almost every $(\mu,\pi)$, by Theorem \ref{thm:hunt-kaloshin}
we have $\dim\pi\mu=\beta$. By Fubini, for a.e. $\pi$ this holds
for a.e. $\mu$, hence $E(\pi)=\beta$ for a.e. $\pi$ (we remark that entropy dimension of a measure is a Borel function of the measure).

We next establish semicontinuity of $E$. Given $\pi\in\Pi_{d,k}$
and $\varepsilon>0$ there is a $q$ so that $E_{q}(\pi)-O(1/q)>E(\pi)-\varepsilon$,
where $O(1/q)$ is the error term in Theorem \ref{thm:linear-projections-of-CP-chains}.
Since $E_{q}$ is continuous, this continues to hold in a neighborhood
$\mathcal{U}$ of $\pi$. By Theorem \ref{thm:linear-projections-of-CP-chains},
for almost every $\mu$, if $\pi'\in\mathcal{U}$ then by letting $q\to\infty$ we get $\dim_{*}\pi'\mu>E(\pi)-\varepsilon$.
This implies that $E(\pi')>E(\pi)-\varepsilon$ for $\pi'\in\mathcal U$, and since $\varepsilon$
was arbitrary, semicontinuity follows.

The last statement follows  from \eqref{eq:lower-bound-for-H-dim-of-projection}.
\end{proof}

We shall also encounter non-ergodic CP-chains. In this case much
of the above fails; for example, the dimension of the projection through
$\pi$ of need not be a.s. constant, as they may differ by ergodic
component (nor do measures for the process have to have the same dimension
a.s.). However, we have the following substitute: \begin{thm} \label{thm:open-dense-set-of-good-directions-non-ergodic-case}Let
$(\mu_{n},B_{n})_{n=1}^{\infty}$ be a CP-chain whose measures almost
surely have exact dimension $\alpha$, and write $\beta=\min\{k,\alpha\}$.
Then for any $\varepsilon>0$ there is an open dense set $\mathcal{U}\subseteq\Pi_{d,k}$
and a set $M_{\varepsilon}$ of measures with $P(M_{\varepsilon})>1-\varepsilon$
and such that \[
\dim_{*}\pi\mu>\beta-\varepsilon\]
 for every $\pi\in\mathcal{U}$ and $\mu\in M_{\varepsilon}$. \end{thm}

\begin{proof} Applying Theorem \ref{thm:hunt-kaloshin} and Fubini,
we can find a dense set of projections $\{\pi_{i}\}_{i=1}^{\infty}\subseteq\Pi_{d,k}$
such that for $P$-a.e. $\mu$, \[
\dim_{*}\pi_{i}\mu=\beta\qquad\mbox{ for }i\in\mathbb{N}.
\]
 This together with equation \eqref{eq:ineq-dimensions} imply that for almost every $\mu$,\[
\liminf_{q\rightarrow\infty}e_{q}(\mu,\pi_{i})\geq\beta\]
for every $i$.

Fix $\varepsilon>0$ and  $i\in\mathbb{N}$. By the previous theorem applied to the ergodic components of $P$, for a.e. every $\mu$ there is an open neighborhood of $\pi_i$ such that for $\pi$ in this neighborhood we have $\dim_* (\pi\mu)>\beta-\varepsilon$ (in fact this neighborhood can be taken to depend only on the ergodic component of $\mu$). It is then clear that there is a set $M_{\varepsilon,i}$ satisfying $P(M_{\varepsilon,i})>1-\varepsilon/2^i$, and an open neighborhood $\mathcal{U}_{\varepsilon,i}$ of $\pi_i$, such that the same holds for $\mu\in M_{\varepsilon,i}$ and $\pi\in\mathcal{U}_{\varepsilon,i}$ (to see this, fix a metric on $\Pi_{d,k}$ and let $r(\mu)>0$ denote the largest number such that the ball of radius $r(\mu)$ around $\pi_i$ has this property. Clearly $r(\mu)$ depends measurably on $\mu$, so, since $r(\cdot)$ is a strictly positive function, we have $\lim_{t\to 0}P(\mu\,:\,r(\mu)>t)=0$. Hence for small enough $t$ we can take $M_{\varepsilon,i}=\{\mu\,:\,r(\mu)>t\}$ and $\mathcal{U}_{\varepsilon,i}$ to be the ball of radius $t$ around $\pi_i$).

Set $M_\varepsilon =\bigcap_{i=1}^\infty M_{\varepsilon,i}$ and $\mathcal{U}=\bigcup_{i=1}^\infty \mathcal{U}_{\varepsilon,i}$; then $P(M_\varepsilon)>1-\varepsilon$ and $\mathcal{U}$ is a dense open neighborhood of $\{\pi_1,\pi_2,\ldots\}$ with the desired properties.
\end{proof}

\subsection{\label{sub:Smooth-images-of-measures}Non-linear images of measures}

Fix an ergodic CP-chain, let $\mu$ be a typical measure for it,
and define $e_{n},E_{n}$ etc. as in the previous section. In the
proof of Theorem \ref{thm:linear-projections-of-CP-chains} linearity
of $\pi$ was used only for the bound\[
\left|H_{\rho^{q}}(\pi\mu^{x,n})-H_{\rho^{n+q}}(\pi\mu_{\mathcal{F}_{n}(x)})\right|<O(1).\]
Now replace $\pi$ with a differentiable but non-linear map $f$. For $n$ large enough, $f|_{\mathcal{F}_n(x)}$ approaches the linear map $D_x f$, and hence, after a little work, the bound above can be replaced by \[
\left|H_{\rho^{q}}((D_x f)\mu^{x,n})-H_{\rho^{n+q}}(f\mu_{\mathcal{F}_{n}(x)})\right|<O(1).\]
We have omitted a few details here but this is essentially how the following proposition is proved.

\begin{prop} \label{prop:smooth-projections-of-CP-chains}Fix
an ergodic CP-chain with distribution $P$, a projection $\pi\in\Pi_{d,k}$,
and define $e_{q}$ and $E_{q}$ as before. Then for all $C^{1}$
maps $\varphi:[0,1]^{d}\rightarrow\RR^{k}$ such that $\sup_{x\in\supp\mu}\|D_{x}\varphi-\pi\|<\rho^{q}$,
we have \[
\dim_{*}(\varphi\mu)\ge E_{q}-O(1/q)
\]
for all $\mu$ that strongly generate the measure component of $P$ along a sequence of partitions $\mathcal{F}_n = \Delta^n(B^*)$ (and in particular for $P$-a.e. $\mu$).
 \end{prop}

\begin{proof} The proof is completely analogous to that of Proposition
\ref{thm:linear-projections-of-CP-chains}, and we only indicate
the differences.

Construct $X$ from $\mathcal{F}$ and $f:X\rightarrow\mathbb{R}^{d}$
precisely as before, lift $\mu$ to $\widetilde{\mu}$ and let $\widetilde{\varphi}=\varphi f:X\rightarrow\mathbb{R}^{k}$.
Construct $X\xrightarrow{g}Y\xrightarrow{h}\mathbb{R}^{k}$ as before.
We wish to estimate the dimension of $g\widetilde{\mu}$ and for this
we must estimate \[
\liminf_{N\rightarrow\infty}\frac{1}{N}\sum_{n=0}^{N-1}H_{\rho^{q(n+1)}}(g\widetilde{\mu}_{[x_{1}\ldots x_{n}]})\]
 for $\widetilde{\mu}$-typical $x\in X$. Briefly, the point is that
as $n\rightarrow\infty$ the map $g$ at $[x_{1}\ldots x_{n}]$ looks
more and more like $\varphi$ on $\mathcal{F}_{qn}(x)$, which looks
more and more like $D_{x}\varphi$, which is not far from $\pi$,
so we are essentially averaging $H_{\rho^{q(n+1)}}(\pi\mu_{\mathcal{F}_{qn}(x)})$.
By $\rho$-regularity of the filtration, this is almost the same as
the average of $H_{\rho^{q}}(\pi T_{\mathcal{F}_{qn}(x)}\mu_{\mathcal{F}_{qn}(x)})$,
which is $H_{\rho^{q}}(\pi\mu^{x,qn})$ and we get our bound from the fact that $\mu$ strongly generates $P$.

Here are the details. Losing an $O(1)$ term, by Theorem \ref{thm:dimension-via-local-entropy-for-rho-trees}
it suffices to estimate \[
\liminf_{N\rightarrow\infty}\frac{1}{N}\sum_{n=0}^{N-1}H_{\rho^{q(n+1)}}(\varphi\mu_{x,qn})\]
 for $\mu$-typical $x$. Now, we change scale: letting $A_{n}$ denote
scaling by $\rho^{-nq}$ on $\mathbb{R}^{k}$, \[
H_{\rho^{q(n+1)}}(\varphi\mu_{x,qn})=H_{\rho^{q}}(A_{n}(\varphi\mu_{x,qn}))\]
 Since $\varphi$ is differentiable at $x$, we see that\[
\lim_{n\rightarrow\infty}\left(A_{n}(\varphi\mu_{x,qn})-A_{n}\circ(D_{x}\varphi)(\mu_{x,qn})\right)=0,\]
 and since $A_{n}\circ D_{x}\varphi=D_{x}\varphi\circ\widetilde{A}_{n}$, where
$\widetilde{A}_{n}$ is scaling by $\rho^{-nq}$ on $\mathbb{R}^{d}$,
we have\[
\lim_{n\rightarrow\infty}\left(A_{n}(\varphi\mu_{x,qn})-(D_{x}\varphi)\circ\widetilde{A}_{n}(\mu_{x,qn})\right)=0.
\]
 The same is true  after we apply $H_{\rho^{q}}$ to these measures,
so it suffices to estimate\[
\liminf_{N\rightarrow\infty}\frac{1}{N}\sum_{n=0}^{N-1}H_{\rho^{q}}((D_{x}\varphi)\circ\widetilde{A}_{n}(\mu_{x,qn})).\]
 Since $\left\Vert D_{x}\varphi-\pi\right\Vert <\rho^{q}$ we have by Lemma \ref{lem:perturbed_entropy} that \[
\left|H_{\rho^{q}}(\pi\nu)-H_{\rho^{q}}((D_{x}\varphi)\nu)\right|=O(1)\]
 for any $\nu$, in particular for $\nu=\widetilde{A}_{n}(\mu_{x,qn})$;
thus we only need to estimate \[
\liminf_{N\rightarrow\infty}\frac{1}{N}\sum_{n=0}^{N-1}H_{\rho^{q}}(\pi\circ\widetilde{A}_{n}(\mu_{x,qn})).\]
Finally, since $H_{\rho^{q}}(\cdot)$ is invariant under translations
and $\mathcal{F}$ is $\rho$-regular, by Lemma \ref{lem:entropy_comparable_radius}
we may replace $\widetilde{A}_{n}$ by $T_{\mathcal{F}_{qn}(x)}$
at the cost of losing another $O(1)$. Since $T_{\mathcal{F}_{qn}(x)}\mu_{x,qn}=\mu^{x,qn}$,
we have reduced the problem to estimating the ergodic averages
\[
\liminf_{N\rightarrow\infty}\frac{1}{N}\sum_{n=0}^{N-1}H_{\rho^{q}}(\pi(\mu_{x,qn}))\]
which, by the fact that $\mu$ strongly generates $P$, is equal to \[
\int H_{\rho^{q}}(\pi\nu)dP(\nu)=q\log(1/\rho)\cdot E_{q}(\pi),
\]
as desired. \end{proof}

\begin{proof}[Proof of Theorem \ref{thm:lower-semicontiunity-in-C1}]
 The theorem is an immediate consequence of Theorem \ref{thm:semicontinuity-1}\eqref{itm:dim-equals-E-ae} and Proposition \ref{prop:smooth-projections-of-CP-chains}.
\end{proof}

We also have the following strengthening of Corollary \ref{cor:open-dense-set}:

\begin{cor} \label{cor:open-good-set-in-C1}
For every $\varepsilon>0$ there is an open set $\mathcal{U}_{\varepsilon}\subseteq C^1([0,1]^d,\RR^k)$ with the following properties:
 \begin{enumerate}
 \item $\mathcal{U}_\e\cap \Pi_{d,k}$ is open, dense and has full measure in $\Pi_{d,k}$.
\item If $\mu$ strongly generates the measure component of an ergodic $CP$-chain along a filtration $\{\Delta^n(B^*)\}$, then
\[
\dim_{*}f\mu>\min(k,\alpha)-\varepsilon \quad\mbox{ for all }f \in \mathcal{U}_{\e},
\]
where $\alpha$ is the dimension of the CP-chain.
 \end{enumerate}
\end{cor}
 \begin{proof}
 This is an immediate consequence of Theorem \ref{thm:semicontinuity-1}\eqref{itm:E-is-ae-constant} and Proposition \ref{prop:smooth-projections-of-CP-chains}.
 \end{proof}

Finally, we obtain Theorem \ref{thm:bound-for-smooth-images} as another consequence of Proposition \ref{prop:smooth-projections-of-CP-chains}.

\begin{proof}[Proof of Theorem \ref{thm:bound-for-smooth-images}]
Fix a measure $\mu$ that generates the measure component of the CP-chain along a filtration $\{\Delta^n(B^*)\}$ and a $C^1$ map $g:\supp\mu\rightarrow\RR^k$ without singular points. It follows from Proposition \ref{prop:smooth-projections-of-CP-chains} that, given a $q\in\NN$, for every $x\in\supp(\mu)$ there is $r=r(x)$ such that
\[
\dim_*(g\mu_{B_r(x)}) \ge E_q(D_x g) - O(1/q).
\]
It easily follows that
\[
\dim_*g\mu \ge \essinf_{x\sim \mu} E_q(D_x g) - O(1/q),
\]
and we obtain the theorem by letting $q\rightarrow\infty$.
\end{proof}

\section{Self-similar measures}

\label{sec:self-similar-measures}

\subsection{\label{sub:self-similar-sets-and-measures}Self-similar sets and measures}

In this section we prove Theorem \ref{thm:projections-of-self-similar-measures-with-irrational-rotation}.
We begin by briefly recalling the main definitions involved.

A map
$f$ on $\RR^{d}$ is called a \emph{contraction} if it is $C$-Lipschitz
for some $C<1$. Let $\Lambda$ be a finite index set; a collection
$\{f_{i}:i\in\Lambda\}$ of contractions on $\RR^{d}$ is called an
\emph{iterated function system} or IFS for short. As is well-known,
there is a unique nonempty compact set $X$, called the \emph{attractor}
of the IFS, such that $X=\bigcup_{i\in\Lambda}f_{i}(X)$. For $a=a_1\ldots a_n \in\Lambda^n$, write \[
  f_a = f_{a_1} \circ f_{a_2} \circ \ldots \circ f_{a_n}.
\]
Given $a\in\Lambda^{\NN}$ and $x\in\RR^{d}$, the sequence
$f_{a_{1}\ldots a_{n}}(x)$ has a limit which lies in
$X$ and is independent of $x$. This defines a continuous and surjective
map $\Phi:\Lambda^{\NN}\rightarrow X$, called the \emph{coding map}.

We say that the \emph{strong separation condition} holds for $\{f_{i}:i\in\Lambda\}$
if the sets $f_{i}(X)$ are pairwise disjoint. This implies that the coding map is injective on the attractor.

Given an iterated function system $\{f_{i}:i\in\Lambda\}$ and
a probability vector $(p_{i})_{i\in\Lambda}$, one can form the product measure on $\Lambda^{\mathbb{N}}$ with marginal $\{p_{i}:i\in\Lambda\}$. The push-forward of this measure is the unique probability measure on $\mathbb{R}^d$ satisfying
\[
\mu=\sum_{i\in\Lambda}p_{i}\, f_{i}\mu.\]
The collection of pairs $\{(f_{i},p_{i})\}_{i\in\Lambda}$ is called
a \emph{weighted iterated function system}.

When all the maps $f_{i}$ of an IFS are contracting similarities, one says
that $X$ is a \emph{self-similar set}, and a measure as above is a \emph{self-similar
measure}. Under the strong separation condition self-similar measures are quite well understood; in particular, they are exact dimensional. In general, however, projections of self-similar sets or measures are not self-similar, and may have complicated overlaps even if the original set does not.

\subsection{\label{sub:semicontinuity-projections-of-self-similar-measures} Proof of Theorem \ref{thm:projections-of-self-similar-measures-with-irrational-rotation}}

Let $\{f_{i}:i\in\Lambda\}$ be an iterated function system satisfying the hypotheses of Theorem \ref{thm:projections-of-self-similar-measures-with-irrational-rotation}. As with Theorem \ref{thm:projection-of-product-invariant-measures}, the proof of Theorem \ref{thm:projections-of-self-similar-measures-with-irrational-rotation} has two parts: first we establish a continuity result (or, rather, a topological statement about the set of nearly-good projections), and then use invariance of (pieces of) the self-similar measure under a sufficiently rich set of orthogonal maps to conclude that all linear projections are good.

We shall rely on the existence of a CP-chain supported on measures closely related to the self-similar measure from which we begin.

\begin{prop} \label{prop:CP-chain-for-self-similar-measure}
Let $\mu$ be a self-similar measure for an IFS satisfying the strong separation property. Then there is a CP-chain supported on measures $\nu$ such that, for some similarities $S,S'$ and Borel sets $B,B'$ (depending on $\nu$), we have $\mu=S\nu_B$ and $\nu=S'\mu_{B'}$.
\end{prop}

There are a number of ways to construct such a CP-chain. One is to rely on Theorem \ref{thm:furstenberg}, but then one must work a bit to show that micromeasures on which the chain is supported have the desired property; not all micromeasures do, since, for example, one can always obtain micromeasures which give positive mass to some affine subspace, even when the original measure gives zero mass to all subspaces. This approach for the homothetic case is discussed in \cite{Gavish09}. The general case is proved in \cite{Hochman09}.

Given a measure $\mu$ on $\RR^d$ and $\alpha\ge 0$, let
\[
\mathcal{U}_\alpha(\mu) = \{ \pi\in \Pi_{d,k} : \dim_*\pi\mu > \alpha \}.
\]
The following observation is immediate from the definition of lower dimension:

\begin{lem} \label{lem:transformation-under-similarity}
Let $\mu, \nu$ be two measures on $\RR^d$, and suppose that $S\mu = \nu_Q$ for some similarity $S$ and some set $Q$ with $\nu(Q)>0$. Let $O$ be the orthogonal part of $S$. Then for any $\alpha$,\[
  \mathcal{U}_\alpha(\mu) \supseteq \{  gO: g\in \mathcal{U}_\alpha(\nu) \}
\]
\end{lem}

With this machinery in place, we have:

\begin{prop} \label{prop:semicontinuity-projections-of-self-similar-measures}
Let $\mu$ be a self-similar measure for an IFS satisfying the strong separation condition. Then for every $\e>0$ there exists an open and dense set $\mathcal{U}\subseteq \Pi_{d,k}$ such that \[
  \dim_{*}(\pi\mu)\ge\min(k,\dim\mu)-\e\quad\text{for all }\pi\in\mathcal{U}.
\]
\end{prop}

\begin{proof}
Choose an ergodic CP-chain supported on measures which, up to similarity, contain a copy of $\mu$ as a restriction. Choose $\varepsilon >0$. By Corollary \ref{cor:open-dense-set}, there is an open dense set of projections $\mathcal{U}'\subseteq\Pi_{d,k}$ so that for a.e. measure for the chain, the image under any $\pi\in\mathcal{U}'$ has dimension at least $\min (k,\dim\mu)-\e$. Choosing a typical measure and applying the previous lemma, we see that \[
  \mathcal{U}_{\alpha-\varepsilon} \supseteq \{\pi O \,:\, \pi\in\mathcal{U}'\},
\]
for some orthogonal map $O$. This completes the proof.
\end{proof}

We first establish Theorem \ref{thm:projections-of-self-similar-measures-with-irrational-rotation} in the case of linear maps:

\begin{prop} \label{prop:almost-linear-projection-good}
For every $\pi\in\Pi_{d,k}$,
\[
\dim(\pi\mu) = \min(k,\dim\mu).
\]
\end{prop}
\begin{proof}
Fix $\varepsilon>0$. By Corollary \ref{cor:proving-expected-dim}, it suffices to show that for every $\pi\in\Pi_{d,k}$, \[
\dim_{*}(\pi\mu)>\min(k,\dim\mu) - \varepsilon.
\]
Let $\mathcal{U}_\varepsilon\subseteq\Pi_{d,k}$ denote the set of projections with this property and let $\pi\in\mathcal{U}_\varepsilon$.  For $a\in\Lambda$ and $B=f_a(X)$, where $X$ is the attractor, the same inequality holds with $\pi\mu_B$ in place of $\pi\mu$. Since $\mu_B = f_a(\mu)$ by strong separation, we conclude that \[
  \dim_* \pi\circ f_a(\mu) > \min(k,\dim\mu) - \varepsilon.
\]
But clearly $\pi\circ f_a\mu$ has the same dimension as $\pi\circ O_a\mu$, where $O_a$ is the orthogonal part of $f_a$, so $\pi\circ O_a\in\mathcal{U}_\varepsilon$. Hence $\mathcal{U}_\varepsilon$ is invariant under the (semi)group action on $\Pi_{d,k}$ generated by pre-composition with$\{O_a\,:\,a\in\Lambda\}$. By Proposition \ref{prop:semicontinuity-projections-of-self-similar-measures} the set $\mathcal{U}_\varepsilon\subseteq\Pi_{d,k}$ has non-empty interior, and by assumption the action in question is minimal; therefore $\mathcal{U}_\varepsilon=\Pi_{d,k}$.
\end{proof}

Finally, to prove the assertion of Theorem \ref{thm:projections-of-self-similar-measures-with-irrational-rotation} about non-linear images of $\mu$, we rely on Theorem \ref{thm:bound-for-smooth-images}. Let $P$ be the distribution of the CP-chain found in Proposition \ref{prop:CP-chain-for-self-similar-measure}. It suffices to show that  the function $E(\pi)$ associated to $P$ in Theorem \ref{thm:semicontinuity} is equal, for every $\pi$, to the expected dimension, i.e. $\alpha=\min\{k,\dim\mu\}$. For this we need only note that, for a $P$-typical measure $\nu$, we have $\nu=S\mu_B$ for some Borel set $B$ and similarity $S$. Since $\dim_*\pi\mu=\alpha$ for every $\pi\in\Pi_{d,k}$, the same holds for $\nu$. Thus the same also holds for entropy dimension $\dim_e\pi\nu$, $\pi\in\Pi_{d,k}$. But since $\nu$ was an arbitrary $P$-typical measure, we see from Theorem \ref{thm:semicontinuity-1} that $E(\pi)=\alpha$ for every $\pi\in\Pi_{d,k}$, as desired.

\section{\label{sec:Measures-invariant-under-xm}Furstenberg's conjecture and measures invariant under
$\times m$}

In this section we prove Theorem \ref{thm:projection-of-product-invariant-measures}:
if $m,n$ are not powers of a common integer
and $\mu$, $\nu$ are respectively $T_{m}$ and
$T_{n}$-invariant measures on $[0,1]$, with $T_{k}x=kx\bmod1$,
then $\mu\times\nu$ projects to a measure of the largest possible
dimension for any projection other than the coordinate projections.

The proof has two parts. The first is to associate to $\mu\times\nu$
a CP-chain and derive topological information about the set of projections
which have the desired property. The second part of the proof uses
irrationality of $\log m/\log n$ to boost this information to the
desired result using a certain invariance of the set of {}``approximately
good'' projections.

\subsection{\label{sub:Invariant-measures-and-CP-chains}Invariant measures
and CP-chains}

We first demonstrate how a $T_{m}$-invariant measure gives rise to
a CP-chain. We do not use this directly here, but it serves
to explain the construction in the next section.

Lift $[0,1]$ to $X^{+}=\{0,\ldots,m-1\}^{\mathbb{N}}$ via base-$m$
coding and denote by $T$ the shift map, which is conjugated to $T_{m}$.
Let $\mu$ be a non-atomic $T_m$ invariant measure on $[0,1]$,
which we identify with its lift to $X^{+}$. Let $\mu$ also denote
the shift-invariant measure on
\[
X=\{0,1,\ldots,m-1\}^{\mathbb{Z}},
\]
obtained as the natural extension of $(X^{+},\mu,T)$; the shift on
$X$ is also denoted $T$. Let $(X_{n})_{n=-\infty}^{\infty}$be the
coordinate functions\[
X_{n}(x)=x_{n}\qquad x\in X\]

Disintegrate $\mu$ with respect to the $\sigma$-algebra $\mathcal{F}^{-}=\sigma(X_{i}\,:\, i\leq0)$.
For $\mu$-a.e. $x\in X$ we obtain the measure $\mu^{x}$ on $X^{+}$,
depending only on $x^{-}=(\ldots x_{-2},x_{-1},x_{0})$, such that
for any $A\subseteq X^{+}$ we have $\mu(A)=\int\mu^{x}(A)d\mu(x)$.
For $\mu$-typical $x\in X$ construct the sequence \[
(\mu_{n},B_{n})=\left(\mu^{T^{n}x},\left[\frac{x_{n-1}}{m},\frac{x_{n-1}+1}{m}\right)\right)\in\mathcal{P}([0,1])\times\mathcal{D}_{m}.\]
Pushing the measure $\mu$ forward via $x\mapsto ((\mu_n,B_n))_{n=1}^{\infty}$ we obtain a stationary $\mathcal{P}([0,1])\times\mathcal{D}_{m}$-valued
process which is seen to be a CP-chain for the base-$m$ partition
operator. This is analogous to the second example in \cite[page 409]{Furstenberg08}.

We record the following well-known fact: If $\mu$ is an ergodic $T_{m}$-invariant
measure, then both $\mu$ and $\mu$-a.e. $\mu^{x}$ are exact dimensional, with
\begin{equation}  \label{eq:entropy-equal-dimension}
\dim\mu = \dim\mu^x = h(\mu,T_{m})/\log m,
\end{equation}
Here $h$ is the Kolmogorov-Sinai
entropy (for $\mu$ this follows from Shannon-McMillan-Breiman, for $\mu^x$ one uses e.g. Lemma \ref{lem:local-entropy-lemma}). When $\mu$ is not ergodic and $\mu=\int\mu_{\omega}d\sigma(\omega)$
is its ergodic decomposition, then $\mu$ is exact dimensional if
and only if almost all ergodic components have the same dimension
(i.e. the same entropy). More generally, one can show that \cite[Theorem 9.1]{LindenstraussPeres99}
\[
h(\mu,T_{m})=\int h(\mu_{\omega},T_{m})d\sigma(\omega)\]
 and \begin{equation} \label{eq:dim-invariant-is-inf-along-components}
\dim_{*}\mu=\essinf_{\omega\sim\sigma}\dim_{*}\mu_{\omega}.
\end{equation}

\subsection{\label{sub:Products-of-xmxn-measures}Products of $\times m$,$\times n$
invariant measures}

We now deal with the more delicate case of a measure \[
\theta=\mu\times\nu\]
where $\mu,\nu$ are measures on $[0,1)$ invariant, respectively,
under $T_{m},T_{n}$. When $\frac{\log m}{\log n}\in\mathbb{Q}$ the
product measure is invariant under the action $T_{k}\times T_{k}$
for some $k\in\mathbb{N}$ which is a common power of $m,n$, and
we get a CP-chain in much the same manner as before. In the case
that $\frac{\log m}{\log n}\notin\mathbb{Q}$ the product action,
and also all actions of products of powers of $T_{m},T_{n}$, are
not local homotheties, and the iterates of the product partitions
do not have bounded eccentricity. Instead, we shall show that this
measure is associated to a CP-chain using a more involved partition operator and some additional randomization.

For concreteness we fix $m=2$ and $n=3$, and assume as we may that
$\mu,\nu$ are non-atomic. As before, identify $\mu,\nu$ with shift-invariant
measures on $X^{+}=\{0,1\}^{\mathbb{N}}$ and $Y^{+}=\{0,1,2\}^{\mathbb{N}}$
and extend to the two sided versions on the corresponding two-sided
subshifts $X,Y$. Let $(X_{n})_{n=-\infty}^{\infty}$, $(Y_{n})_{n=-\infty}^{\infty}$
denote the coordinate functions, and let $\mu^{x},x\in X$ and $\nu^{y},y\in Y$
denote the disintegrations with respect to $\mathcal{F}^{-}=\sigma(X_{n}\,:\, n\leq0)$
and $\mathcal{G}^{-}=\sigma(Y_{n}\,:\, n\leq0\}$, respectively.

To construct our CP-chain we first describe our partition operator.
For $w\in[0,\log3)$ define the rectangle \[
R_{w}=[0,1]\times[0,e^{w}]\]
 and let $\mathcal{E}$ be all rectangles similar to some $R_{w},w\in[0,\log3)$.
We shall define the partition operator $\Delta$ on the sets $R_{w}$
and extend by similarity to the rest of $\mathcal{E}$.

To apply $\Delta$ to $R_{w}$, first split $R_{w}$ into $R'=[0,\frac{1}{2}]\times[0,e^{w}]$
and $R''=(\frac{1}{2},1]\times[0,e^{w}]$. Then, if $w>\log(3/2)$,
split $R'$ into three sub-rectangles with the same base and heights
$\frac{1}{3}e^{w}=e^{w-\log3}$, and similarly $R''$. The partition
obtained is $\Delta(R_{w})$; see Figure \ref{fig:partition}. It is a partition of $R_{w}$ into either
two or six copies of $R_{w'}$, where \begin{equation}
w'=w+\log2-1_{\{w\geq\log3-\log2\}}\cdot\log3\in[0,\log3)\label{eq:irrational-rotation-in-x2x3-measures}\end{equation}

\begin{figure}
\begin{center}
\includegraphics[width=0.8\textwidth]{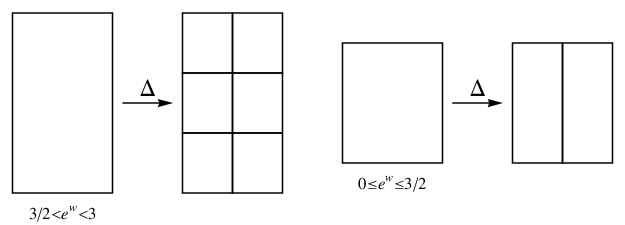}
\end{center}
\caption{The partition operator $\Delta$. If the eccentricity of $R$ is $e^w$, then for the rectangles in $\Delta(R)$ the eccentricity is $e^{w'}$ where $w' = w+\log 2 \bmod\log 3$.} \label{fig:partition}
\end{figure}

 Thus, applying $\Delta$ repeatedly to a rectangle $R_{w_{1}}$ results
in partitions into rectangles similar to $R_{w_{n}}$, $n=1,2,3\ldots$
and the sequence $(w_{n})_{n=1}^{\infty}$ evolves according to an
irrational rotation by $\log2/\log3$.

\begin{lem} \label{lem:regularity-of-irrational-partition}For any
$w\in[0,\log3)$, the filtration $\mathcal{F}_{n}=\Delta^{n}R_{w}$
is $\frac{1}{2}$-regular.\end{lem}

\begin{proof} Immediate, since the base of rectangles in $\Delta^{n}R_{w}$
are of length $(\frac{1}{2})^{n}$ and by definition their eccentricity
is uniformly bounded. \end{proof}

Let \[
S_{w}=\left(\begin{array}{cc}
e^{-w/2} & 0\\
0 & e^{w/2}\end{array}\right),\]
 so $S_{w}([0,1]^{2})=R_{w}^{*}$ and the axes are eigen-directions.
Choose $(x,y,w)\in X\times Y\times[0,\log3)$ according to $\mu\times\nu\times\lambda$,
where $\lambda$ is normalized Lebesgue measure, and associate to it the measure
\[
\tau=\tau_{x,y,w}=S_{w}(\mu^{x}\times\nu^{y}).
\]
The distribution of $\tau$ is the marginal of the measure component
of a CP-chain. To get the distribution on pairs which is the marginal
of the full CP-chain, choose $B'\in\Delta(R_{w})$ with weights
$\tau(B')$ and setting $\tau'=\tau^{B'}$; then $(\tau',B')$ is
the desired marginal (or course, $\tau'$ has the same distribution
as $\tau$). One verifies this by checking that the distribution is
fixed by the transition law with respect to $\Delta$.%
\footnote{The dynamics of this process may also be described as follows. Let
$\Sigma=X\times Y\times[0,\log3)$ with the measure $\mu\times\nu\times\lambda$.
Let $T$ denote the shift on $X,Y$. Define a transformation $S:\Sigma\rightarrow\Sigma$
by\[
S(\omega,\omega',w)=(T\omega,T^{\{w\geq\log3-\log2)}\omega',w+\log2\bmod\log3)\]
 where $T^{\sigma}=T$ if the event $\sigma$ occurs and is the identity
otherwise. This represents the dynamics as a skew product $X\times[0,\log3)$
with fiber $Y$, in which the second coordinate $w$ is used to control
the average speed with which the third coordinate (i.e. $Y$) is advanced.%
}

Note that, if $\mu$ and $\nu$ are ergodic, then $S_w(\mu^{x}\times\nu^{y})$ almost surely has the same dimension
as $\theta=\mu\times\nu$.

\begin{prop} \label{prop:large-set-of-good-dirs-for-x2x3-product}Suppose $\mu$ and $\nu$ are ergodic. For
every $\varepsilon>0$ there is a subset $A_{\varepsilon}\subseteq[0,1]^{2}$
with $\theta(A_{\varepsilon})>1-\varepsilon$ and a dense, open set
$\mathcal{U}_{\varepsilon}\subseteq\Pi_{2,1}$ such that for $\pi\in\mathcal{U}_{\varepsilon}$,
\[
\dim_*\pi(\theta|_{A_{\varepsilon}})>\min\{1,\dim\theta\}-\varepsilon.\]
 \end{prop}

\begin{proof} Let $P$ denote the distribution of the CP-chain,
and write \[
\alpha=\min\{1,\dim\theta\}.\]
 This is the expected dimension of the projection of a typical measure for
a typical ergodic component of the process. By Theorem \ref{thm:open-dense-set-of-good-directions-non-ergodic-case}
we can find a dense open set $\mathcal{U}'_{\varepsilon}\subseteq\Pi_{2,1}$
and $A'_{\varepsilon}\subseteq\mathcal{P}(\mathbb{R}^{d})$ such that
$P(A'_{\varepsilon})>1-\varepsilon$, and $\dim_{*}\pi\tau>\alpha-\varepsilon$
for all $\pi\in\mathcal{U}'_{\varepsilon}$ and $\tau\in A'_{\varepsilon}$.

Using the fact that, conditioned on $w$, the distribution of $\tau=\tau_{x,y,w}$
is $S_{w}(\mu^{x}\times\nu^{y})$, with $(x,y)\sim\theta$, we can
apply Fubini to find a $w_{0}\in[0,\log3)$ and a Borel set $E\subseteq X\times Y$
such that $\theta$$(E)>1-\varepsilon$ and for every $(x,y)\in E$,
\[
\dim_{*}\pi(S_{w_{0}}\mu^{x}\times\nu^{y})>\alpha-\varepsilon\qquad\mbox{ for all }\pi\in\mathcal{U}'_{\varepsilon}.\]
For an affine map $S$ we identify $\pi\circ S$ with the projection $\pi'$ whose pre-image sets partition $\mathbb{R}^2$ into the same lines as $\pi\circ S$. Let \[
\mathcal{U}_{\varepsilon}=\{\pi\circ S_{w_{0}}\,:\,\pi\in\mathcal{U}'_{\varepsilon}\}.\]
 Then $\mathcal{U}_{\varepsilon}$ is open and dense in $\Pi_{2,1}$
and for $(x,y)\in E$, \[
\dim_{*}\pi(\mu^{x}\times\nu^{y})>\alpha-\varepsilon\qquad\mbox{ for all }\pi\in\mathcal{U}_{\varepsilon}\]
Define $\eta\in\mathcal{P}([0,1]^{2})$ by \[
\eta=\int_{E}\mu^{x}\times\nu^{y}\; d\theta(x,y)\]
 Then by Lemma \ref{lem:properties-of-lower-dim}, \[
\dim_{*}\pi\eta>\alpha-\varepsilon\qquad\mbox{ for all }\pi\in\mathcal{U}_{\varepsilon}.\]
 Since \[
\theta=\mu\times\nu=\int\mu^{x}\times\nu^{y}\; d\theta(x,y),\]
 we have $\eta\ll\mu\times\nu$, and since $\theta(E)>1-\varepsilon$
there is a subset $A_{\varepsilon}\subseteq[0,1]^{2}$ such that $\theta|_{A_{\varepsilon}}\sim\eta$
and $\theta(A_{\varepsilon})>1-\varepsilon$. This is the desired
set. \end{proof}

\subsection{\label{sub:Proof-of-x2x3}Proof of Theorem \ref{thm:projection-of-product-invariant-measures}}

Let $\mu,\nu$ be ergodic $T_{2},T_{3}$-invariant measures on $[0,1]$,
respectively, such that $\dim_{*}\mu,\dim_{*}\nu>0$. Let $\theta=\mu\times\nu$.
Write\[
\alpha=\min\{1,\dim\theta\},\]
 and for $\varepsilon>0$ let $\mathcal{V}_{\varepsilon}\subseteq\Pi_{2,1}$
 denote the set \[
\mathcal{V}_{\varepsilon}=\left\{ \pi\in\Pi_{2,1}\,:\,\begin{array}{c}
\mbox{There exists }A\subseteq[0,1]^{2}\mbox{ with }\\
\theta(A)>1-\varepsilon\mbox{ and }\dim_*\pi(\theta|_{A})>\alpha-\varepsilon\end{array}\right\} \]

In the previous section we saw that the interior of $\mathcal{V}_{\varepsilon}$
is open and dense in $\Pi_{2,1}$. We now establish some invariance.
For this it is convenient to represent projections $\pi\in\Pi_{2,1}$
by the slope of their image, i.e. $\pi_{a}$ projects onto the line
$y=ax+b$ (we cannot represent projection to the $y$-axis, but also
do not want to). We note that this identification of $\mathcal{V}_{\varepsilon}$
as a subset of $\mathbb{R}$ is consistent with the topology on $\Pi_{2,1}$.
\begin{prop} \label{prop:invariance-of-almost-good-directions}$\mathcal{V}_{\varepsilon}$
is invariant under the action by multiplication of the semigroup $\mathcal{S}=\{\frac{3^{i}}{2^{j}}\,:\, i,j\in\mathbb{N}\}\subseteq(\mathbb{N},\times)$,
i.e. if $\pi_{a}\in\mathcal{V}_{\varepsilon}$ then $\pi_{3a,}\pi_{a/2}\in\mathcal{V}_{\varepsilon}$.\end{prop}
\begin{proof} Suppose $\pi_{a}\in\mathcal{V}_{\varepsilon}$ and
let $A\subseteq[0,1]^{2}$ be a set of measure $>1-\varepsilon$ as
in the definition of $\mathcal{V}_{\varepsilon}$.

Write $T_2\times \id$  for the map $[0,1]^2\to[0,1]^2$ given by \[
(T_{2}\times\id)(x,y)=(T_{2}x,y),\]
 and let\begin{eqnarray*}
A_{0} & = & (T_{2}\times\id)^{-1}A\cap([0,\frac{1}{2})\times[0,1]),\\
A_{1} & = & (T_{2}\times\id)^{-1}A\cap([\frac{1}{2},1)\times[0,1]),\end{eqnarray*}
 so $(T_{2}\times\id)^{-1}A=A_{0}\cup A_{1}$. By the invariance properties
of $\mu,\nu$ we see that $\mu\times\nu$ is invariant under $T_2\times\id$,
so \[
\mu\times\nu(A_0\cup A_1)=\mu\times\nu((T_{2}\times\id)^{-1}A)>1-\varepsilon.\]
 Also, write\begin{eqnarray*}
\mu_{0} & = & \mu|_{[0,1/2)\times[0,1]}\\
\mu_{1} & = & \mu|_{[1/2,1]\times[0,1]}\end{eqnarray*}

Now, by invariance we have\begin{equation}
\mu\times\nu=(T_{2}\mu_{0}\times\nu)+(T_{2}\mu_{1}\times\nu)\label{eq:decomposition-of-preimage-of-2x3-measure}\end{equation}
It follows that there are affine maps $S_{0},S_{1}$
of $\mathbb{R}$ with \[
\pi_{a}(\mu\times\nu|_{A})=S_{0}\pi_{a/2}(\mu_{0}\times\nu|_{A_{0}})+S_{1}\pi_{a/2}(\mu_{1}\times\nu|_{A_{1}})\]
 One way to see this is to note that the fibers of $\pi_{a/2}$ are
lines of slope $-2/a$ and are mapped under $(x,y)\mapsto(2x,y)$
to the fibers of $\pi_{a}$, which have slope $-1/a$.

By assumption $\dim_*\pi_a(\mu\times\nu|_A)>\alpha-\varepsilon$. Affine
maps preserve dimension, so \begin{eqnarray*}
\alpha-\varepsilon & < & \dim_*\pi_{a}(\mu\times\nu|_{A})\\
 & = & \min\left\{ \dim_* S_{0}\pi_{a/2}(\mu_{0}\times\nu|_{A_{0}}),\dim_* S_{1}\pi_{a/2}(\mu_{1}\times\nu|_{A_{1}})\right\} \\
 & = & \min\left\{ \dim_*\pi_{a/2}(\mu_{0}\times\nu|_{A_{0}}),\dim_*\pi_{a/2}(\mu_{1}\times\nu|_{A_{1}})\right\} \\
 & = & \dim_*\pi_{a/2}(\mu\times\nu|_{(T_2\times \id)^{-1}A}),\end{eqnarray*}
 whence $\pi_{a/2}\in\mathcal{V}_{\varepsilon}$.

A similar analysis, the measures  $\nu_{i}=\nu|_{[0,1]\times[i/3,(i+1)/3)}$ and the identity \[
\mu\times\nu=T_{3}(\mu\times\nu_{0})+T_{3}(\mu\times\nu_{1})+T_{3}(\mu\times\nu_{2})\]
show that if $\pi_{a}\in\mathcal{V}_{\varepsilon}$, then $\pi_{3a}\in\mathcal{V}_{\varepsilon}$. \end{proof}

We can now conclude the proof of Theorem \ref{thm:projection-of-product-invariant-measures}:

\begin{proof}[Proof of Theorem \ref{thm:projection-of-product-invariant-measures}]
Using \eqref{eq:dim-invariant-is-inf-along-components} for $\mu,\nu$ and Lemma \ref{lem:properties-of-lower-dim}\eqref{itm:dim-at-least-inf-over-components} applied to $\pi(\mu\times\nu)$, the general case is reduced to the case in which $\mu$ and $\nu$ are ergodic. Furthermore, by Corollary \ref{cor:proving-expected-dim} it is enough to show that $\dim_*\pi(\mu\times \nu) \ge \alpha$.

Since $\log3/\log2\notin\mathbb{Q}$, the semigroup \[
\mathcal{S}=\left\{\frac{3^{i}}{2^{j}}\,:\, i,j\in\mathbb{N}\right\}
\]
 is dense in $\mathbb{R}^{+}$ \cite{Furstenberg67}; this is the only place in the proof where we use the arithmetic properties of $m,n$. For each $\varepsilon>0$
we have seen that the interior of $\mathcal{V}_{\varepsilon}$ is
open and dense, in particular contains an open set in each of the
rays $\mathbb{R}^{+},\mathbb{R}^{-}$; and since it is invariant under
$\mathcal{S}$ we see that \[
\mathcal{V}_{\varepsilon}=\{\pi_{a}\,:\, a\neq0\}.\]
 For any $\pi_a$, $a\neq0$ we have thus shown that there is a set $A_{\varepsilon}$
with $\mu\times\nu(A_{\varepsilon})>1-\varepsilon$ and\[
\dim_{*}\pi_a(\mu\times\nu|_{A_{\varepsilon}})>\alpha-\varepsilon,\]
 which implies\[
\dim_{*}\pi(\mu\times\nu)\geq\alpha,\]
 as desired. \end{proof}

\subsection{\label{sub:Proof-of-topological-conj}Proof of the topological conjecture}

We briefly show how to derive the topological version, conjecture
\ref{con:Furstenberg-conjecture}, from the measure one above. Suppose
that $X,Y\subseteq[0,1]$ are closed and invariant under $T_{2},T_{3}$,
respectively. Using the variational principle (see e.g. \cite[Theorem 14.1]{PollicottYuri})
we can find probability measures $\mu,\nu$ on $[0,1)$, invariant
and ergodic under $T_{2},T_{3}$, respectively, such that $\dim\mu=\dim A$
and $\dim\nu=\dim B$.

Let $\theta=\mu\times\nu$, so that $\dim\theta=\dim A\times B$.
For any projection $\pi\in\Pi_{2,1}\setminus\{\pi_x,\pi_y\}$, we have by the theorem above that \[
\dim\pi\theta=\min\{1,\dim\theta\}\]
 Since $\pi\theta$ is supported on $\pi(A\times B)$, we have \[
\dim\pi(A\times B)\geq\min\{1,\dim\theta\}=\min\{1,\dim(A\times B)\},\]
 The right hand side is also an upper bound, so we are done.

\subsection{\label{sub:Rudolph-Johnson-theorem}The Rudolph-Johnson theorem}

In this section we show how Theorem \ref{thm:projection-of-product-invariant-measures}
implies the Rudolph-Johnson theorem. We first prove the theorem under the hypothesis that the measure has positive dimension rather than positive entropy of ergodic components.

\begin{thm} Let $\mu$ be a probability measure on $[0,1]$ that is invariant under $T_{m}$ and $T_{n}$, with $m,n$ not powers of the same integer. Suppose that $\dim_*\mu>0$. Then $\mu$=Lebesgue measure. \end{thm}

\begin{proof}
Write $\alpha=\dim_{*}\mu>0$. Denote by $*$ the convolution in $\mathbb{R}$ and by $\circ$ the convolution
in $\mathbb{R}/\mathbb{Z}$. Notice that, up to an affine map, $\mu*\mu$
and $\pi_{1}(\mu\times\mu)$ are the same (recall that $\pi_{1}$
is projection to the line $y=x$). Thus, by Theorem \ref{thm:projection-of-product-invariant-measures}
with $\mu=\nu$, \[
\dim_*(\mu*\nu)=\dim_*\pi_{1}(\mu\times\nu)=\min\{1,2\alpha\}\]
 Also, since\[
\mu\circ\mu=\mu*\mu\bmod1\]
 and reduction modulo $1$ is a countable-to-one local isometry,\[
\dim_{*}(\mu\circ\mu)=\min\{1,2\alpha\}.\]
 Finally, it is easy to check that $\mu\circ\mu$ is again invariant
under $T_{m}$ and $T_{n}$, because $T_{m},T_{n}$ are endomorphisms
 of $\mathbb{R}/\mathbb{Z}$.

Iterating this argument, we see that for each $k$,\[
\dim_{*}\underset{2^{k}\mbox{ times}}{\underbrace{\mu\circ\mu\circ\ldots\circ\mu}}=\min\{1,2^{k}\alpha\},\]
 and since $\alpha>0$ there is a $k$ so that, for $\nu=\circ^{2^{k}}\mu$,\[
\dim_{*}\nu=1\]
 But $\nu$ is a $T_{2}$-invariant measure, and by \eqref{eq:entropy-equal-dimension} and \eqref{eq:dim-invariant-is-inf-along-components} Lebesgue measure
is the only $T_{m}$ invariant measure of lower dimension $1$. Thus
\[
\nu=\lambda.\]
 To establish $\mu=\lambda$, we look at the Fourier coefficients,
where convolution translates to multiplication. For $i\neq0$ we have
\[
\widehat{\mu}(i)^{2^{k}}=\widehat{\nu}(i)=\widehat{\lambda}(i)=0\]
 Thus $\widehat{\mu}(i)=0$ for $i\neq 0$, so $\mu=\lambda$.
\end{proof}

For measures which are ergodic under $T_m$ or even under the joint action of $T_m,T_n$, positive dimension and positive entropy are equivalent conditions. However, the condition that all ergodic components have positive entropy is weaker; entropy behaves like the mean of the entropies of ergodic components, while dimension behaves like the essential infimum. Nevertheless, there is a simple reduction which recovers the non-ergodic case.

\begin{thm}
[Rudolph-Johnson \cite{Rudolph90,Johnson92}]
Let $\mu$ be a probability measure on $[0,1]$ that is invariant under $T_{m}$ and $T_{n}$, with $m,n$ not powers of the same integer. Suppose that all ergodic components of $\mu$ with respect to $T_m$ have positive entropy. Then $\mu$=Lebesgue measure. \end{thm}

\begin{proof} Let $\mu=\int \mu_\omega d\nu(\omega)$ denote the ergodic decomposition of $\mu$ with respect to $T_m$, where $\nu$ is defined on an auxiliary space $\Omega$. Since $T_n$ is an endomorphism of $([0,1],\mu,T_m)$ it acts on the ergodic components, i.e. on $(\Omega,\nu)$, in a measure-preserving manner. Furthermore, since the joint action of $T_m,T_n$ on $\mu$ is ergodic, $T_n$ acts ergodically on the space $(\Omega,\nu)$ of ergodic components.

Fix $t>0$, let $\Omega_t = \{\omega\in\Omega \,:\, h(T_m,\mu_\omega)\geq t\}$, and let $\mu_{\geq t}=\int_{\Omega_t} \mu_\omega d\nu(\omega)$ . Since $T_n$ is a factor map between $([0,1],T_m,\mu_\omega)$ and $([0,1],T_m,T_n\mu_\omega)$, and since factor maps don't increase entropy, we have $h(T_m,\mu_\omega)\geq h(T_m,T_n\mu_\omega)$, and therefore $T_n^{-1}\Omega_t \subseteq\Omega_t$. Since $T_n$ is $\nu$-preserving, $\Omega_t$ is $T_n$-invariant up to measure $0$.

Hence $\mu_{\geq t}$ is $T_n$-invariant. It is also clearly $T_m$ invariant, and its dimension is \[
  \dim_* \mu_{\geq t} = \essinf_{\omega\in\Omega_t} \frac{h(\mu_\omega,T_m)}{\log m} \geq \frac{t}{\log m} > 0.
\]
Hence by the previous theorem $\mu_{\geq t}$ is Lebesgue measure. Since $\mu=\lim\mu_{\geq t}$, we are done.
\end{proof}

\section{Convolutions of Gibbs measures}

\label{sec:convolutions-Gibbs-measures}

\subsection{\label{sub:Gibbs-preliminaries}Preliminaries}

Let $\mathcal{I}=\{f_{i}:i\in\Lambda\}$ be an iterated function system
on the interval $[0,1]$. The IFS $\mathcal{I}$
is called a \emph{regular IFS} if the following conditions hold:
\begin{enumerate}
\item Regularity: There is $\e>0$ such that each $f_{i}$ is a $C^{1+\e}$
map on a neighborhood of $[0,1]$.
\item Contraction and orientation: $0<Df_{i}(x)<1$ for all $i$ and all
$x\in[0,1]$.
\item Separation: The sets $f_{i}((0,1))$ are pairwise disjoint
subsets of $(0,1)$.
\end{enumerate}
We say that a closed set $X\subseteq[0,1]$ is a \emph{regular
Cantor set} if it is the attractor of a regular IFS. This definition
is more restrictive than the one in e.g. \cite{MoreiraYoccoz01}. Our methods can handle the more general setting with slight modifications, but for simplicity we concentrate on the case above.

We recall the definition and some basic facts about Gibbs measures;
a clearly written introduction to this topic can be found in \cite[Chapter 5]{Falconer97}.
Let $\Lambda$ be a finite set. If $\varphi:\Lambda^{\NN}\rightarrow\mathbb{R}$
is a H\"{o}lder-continuous function, then there exist a unique real number
$P(\varphi)$ and a unique ergodic measure $\mu_{\varphi}$ on $\Lambda^{\NN}$,
such that \begin{equation}
\mu_{\varphi}([x_{1}\ldots x_{n}])=\Theta_{\varphi}\left(\exp\left(-nP(\varphi)+\sum_{j=0}^{n-1}\varphi(T^{j}x)\right)\right),\label{eq:gibbs}\end{equation}
 for all $x\in\Lambda^{\NN}$, where $T$ is the shift on $\Lambda^\NN$. (Recall that $A=\Theta_{\varphi}(B)$ means that $C^{-1}A \le B\le C A$ for a constant $C>0$ depending only on $\varphi$.) The number $P(\varphi)$ is
known as the \emph{topological pressure} of $\varphi$, and the
measure $\mu_{\varphi}$ as the \emph{Gibbs measure for the potential}
 $\varphi$.

We say that two measures defined on the same measure space are
$C$-\emph{equivalent} if they are mutually absolutely continuous
and both Radon-Nikodym derivatives are bounded by $C$; this is denoted by $\sim_{C}$. If $x,y$ are numbers, we also write
$x\,\sim_{C}\, y$ to denote that $x\le Cy$ and $y\le Cx$.

Given a cylinder set $[a]$ and a measure $\mu$ on $\Lambda^{\mathbb{N}} $, let $\mu^{[a]}$ denote the probability measure on $\Lambda^{\mathbb{N}}$ given by \[
  \mu^{[a]}([b])=\frac{1}{\mu([a])}\mu([ab]).
\]
This is the symbolic analogue of rescaling a measure on $\mathbb{R}^d$.

While Gibbs measures are not generally product measures, they do satisfy a slightly weaker property, which is the only one of their properties that we  use:

\begin{lem} Let $\mu$ be a Gibbs measure on $\Lambda^{\NN}$ for
some H\"{o}lder potential $\varphi$. Then there is $C=C(\varphi)>0$
such that for any word $a\in\Lambda^*$, $\mu\sim_{C}\mu^{[a]}$.
\end{lem} \begin{proof} It follows from \eqref{eq:gibbs} that \[
\mu^{[a]}[b]=\frac{\mu[ab]}{\mu[a]}=\Theta_{\varphi}(\mu[b]).\]
 This shows that $\mu^{[a]}$ and $\mu$ are $C$-equivalent for
some $C$ that depends on $\varphi$ only.
\end{proof}

Motivated by the previous lemma, we say that a measure $\mu$
on $\Lambda^{\NN}$ is a $C$-\emph{quasi-product measure} if $C>0$ and $\mu^{[a]}\,\sim_{C}\,\mu$ for
all $a\in\Lambda^*$ or, equivalently, \[
C^{-1}\mu([a b])\le\mu([a])\mu([b])\le C\mu([ab]).\]

A few comments about this notion are in order. It is easy to see that the support of a quasi-product measure on
$\Lambda^{\NN}$ is $\Xi^{\NN}$ for some subset $\Xi\subseteq\Lambda$.
Thus by replacing $\Lambda$ by $\Xi$ we can always assume that quasi-product
measures are globally supported. Each $T^n \mu$ is $C$-equivalent to $\mu$, and taking any weak-* limit point $\nu$ of $\frac{1}{n}\left(\mu+T\mu+T^2\mu+\ldots+ T^{n-1}\mu\right)$, as $n\to\infty$, we obtain a $T$-invariant measure equivalent to $\mu$ and with Radon-Nikodym derivative bounded between $C^{-1}$ and $C$. Also, it is easy to show that $\mu$ has a rudimentary mixing property: there is a constant $K$ such that if $A,B$ are sets then $\mu(T^{-n} A\cap B)>K\mu(A)\mu(B)$ for large enough $n$. It follows that the same is true of $\nu$, perhaps with a different constant, and so $\nu$ is ergodic. It follows now that $\nu$, and hence $\mu$, are exact dimensional.

Finally, we transfer these notions to the geometric attractor: A Gibbs measure on $X$ is the projection of a Gibbs measure
on $\Lambda^{\NN}$ under the coding map. In the same way we define
quasi-product measures on $X$. Thus, Gibbs measures on $X$ are quasi-product
measures.

We have the following slight strengthening of Theorem \ref{thm:convolutions-of-Gibbs-measures}:

\begin{thm} \label{thm:convolutions-of-quasi-product-measures} Theorem \ref{thm:convolutions-of-Gibbs-measures} holds for quasi-product measures $\mu_i$ on the attractors $X_i$.
\end{thm}

For notational simplicity we prove Theorem \ref{thm:convolutions-of-quasi-product-measures} for $d=2$; the proof is the same in higher dimensions. The structure of the proof resembles that of Theorem
\ref{thm:projections-of-self-similar-measures-with-irrational-rotation}, but is technically more difficult. The main differences are, first, that micromeasures of Gibbs measures on regular IFSs are harder to relate to the original measure; and, second, that moving the open set around $\Pi_{2,1}$ is more involved due to nonlinearity of the IFSs. In the next section we present some standard tools for dealing with these problems.

Before embarking on the proof we explain its relation to Moreira's proof of \eqref{eq:projections-of-sets} in the strictly non-linear
case. The key device there was the so-called Scale Recurrence
Lemma (SRL) of \cite{MoreiraYoccoz01}. This technical result, which relies on the non-linearity of the system, gives information about the quotients
\[
\frac{\diam(f_{a}^{(1)}([0,1]))}{\diam(f_{b}^{(2)}([0,1]))}\]
 for many (but not all) pairs $a\in\Lambda_1^*,b\in\Lambda_2^*$. In the proof one uses only the pairs of words that are ``good'' in the sense of the
SRL in order to construct subsets $X'_{i}\subseteq X_{i}$
which approximate $X_{i}$ well in dimension, but have zero measure
with respect to any Gibbs measure. Hence it seems unlikely that Moreira's
proof can be adapted to give information about convolutions of Gibbs
measures. Our proof uses the machinery of CP processes in place of the SRL, specifically in the proof of Proposition \ref{prop:semicontinuity-convolutions-of-quasi-product-measures}. After that our proof follows Moreira's original argument.

\subsection{Limit geometries and micromeasures}

Throughout this sub-section, $\mathcal{I}=\{f_{i}:i\in\Lambda\}$
is a regular IFS with attractor $X$, and $\mu$ is a $C$-quasi-product
measure on $X$. By definition, $\mu=\Phi\eta$, where $\Phi$ is the
coding map and $\eta$ is a $C$-quasi-product measure on $\Lambda^{\NN}$.

Given a closed interval $[a,b]$,
denote by $T_{[a,b]}$ the unique orientation-preserving
affine map sending $[a,b]$ to $[0,1]$. If $f:[0,1]\rightarrow\RR$
is a continuous injection, write \[
  T_{f}=T_{f([0,1])}
\] and \[
  f^* = T_f f.
\]
Hence $f^*$ is a bijection of $[0,1]$. Also,  for $a\in\Lambda^*$ write \[
  I(a)=f_{a}([0,1]).
\]The following simple lemma is a consequence of the separation assumption on the IFS and $C$-equivalence.

\begin{lem} \label{lem:restrictions-of-quasi-product-measures} For
every $a\in\Lambda^{*}$, \[
\mu_{I(a)}\,\sim_{C}\, f_{a}(\mu).\]
 \end{lem}
 \begin{proof}
For any $b\in\Lambda^{*}$, \[
\mu_{I(a)}(I(ab))=\frac{\mu(I(ab))}{\mu(I(a))}\,\sim_{C}\,\mu(I(b))=\left(f_{a}\mu\right)(I(ab)).\]
 Since $\{I(ab)\}$ is a basis of closed sets of $X\cap I(a)$,
this yields the lemma.
\end{proof}

The following result goes back to Sullivan \cite{Sullivan88},
with an explicit proof given by Bedford and Fisher \cite{BedfordFisher97}.
It will be a key tool in the proof, in particular allowing us to describe the micromeasures of a Gibbs measure.

\begin{thm} \label{thm:sullivan} For every left-infinite sequence $x=(\ldots x_{-2}x_{-1})$,
the sequence of rescaled diffeomorphisms $ (f_{x_{-n}} \cdots f_{x_{-1}})^*$ converges in the $C^{1}$ topology to a diffeomorphism $F_{x}:[0,1]\rightarrow[0,1]$,
uniformly in $x$. In particular, the map $x\mapsto F_x$ is continuous. \end{thm}

In this case, the sets $F_{x}(X)$ are known as \emph{limit
geometries} of $X$.
We  also refer to the maps $F_{x}$
themselves as \emph{limit diffeomorphisms}.

Heuristically, the reason
for the validity of Sullivan's Theorem is the following: when composing
$f_{x_{-n}}\circ\cdots\circ f_{x_{-1}}$, the strongest nonlinear distortion
comes from the first map applied, $f_{x_{-1}}$. Since each map is a contraction,
$f_{x_{-i}}$ is applied to an interval of length exponentially small
in $i$. As $i$ gets large, the nonlinear effect of $f_{x_{-i}}$ becomes
negligible (the details of this argument require $C^{1+\e}$-regularity of the contractions).

Note that when the contractions are similarities the maps $F_{a}$ are
all the identity, and Sullivan's Theorem becomes trivial.

If $x\in\Lambda^{-\NN}$, one may consider
the conjugated IFS \[
\mathcal{I}_{x}=\{F_{x}f_{i}F_{x}^{-1}:i\in\Lambda\},\]
whose attractor is the limit geometry $F_x(X)$. More generally, if $\Phi_x$ is the coding map for $\mathcal{I}_{x}$, then $\Phi_x=F_x\Phi$.

\begin{cor} \label{cor:consequences-Sullivan}
\begin{enumerate}
\item[(i)] \label{cor:conjugate-to-get-linear-map}
Choose $a\in\Lambda^*$, and write $\overline{a}=(\ldots a a)\in\Lambda^{-\NN}$. Then  $F_{\overline{a}}f_{a}F_{\overline{a}}^{-1}$
is an affine map, and $\lambda(F_{\overline{a}}f_{a}F_{\overline{a}}^{-1})=\lambda(f_a)$. (Recall \eqref{eq:def-lambda} for the definition of $\lambda(\cdot)$.)
\item[(ii)] Let $g\in C^1([0,1],\RR)$, and consider the conjugated IFS $\mathcal{I}' = \{ g f_i g^{-1} :i\in\Lambda\}$. Then for any $x\in\Lambda^{-\NN}$, $\mathcal{I}_x = \mathcal{I}'_x$. In particular, this holds when $g$ is itself a limit diffeomorphism $F_y$.
\end{enumerate}
\end{cor}
\begin{proof} (i) Write $T_{n}=T_{f_{a^n}}$, where $a^n$ is the $n$-fold concatenation of $a$.
We have \begin{align*}
F_{\overline{a}}f_{a}F_{\overline{a}}^{-1} & =\lim_{n}T_{n}f_{a^{n}}f_{a}F_{\overline{a}}^{-1}\\
 & =\lim_{n}\left(T_{n}T_{n+1}^{-1}\right)T_{n+1}f_{a^{n+1}}F_{\overline{a}}^{-1}\\
 & =\lim_{n}T_{n}T_{n+1}^{-1},\end{align*}
 which is affine as a limit of affine maps. The last assertion follows since $\lambda(\cdot)$ is conjugacy-invariant.

(ii) Limit diffeomorphisms are clearly invariant under affine changes of coordinates and, being an infinitesimal property of the IFS, therefore also under smooth changes of coordinates.
 \end{proof}

Next, we state a well-known consequence of the principle of bounded
distortion, which can be easily deduced from Theorem \ref{thm:sullivan}. For the rest of the section, the constants implicit in the $O(\cdot)$ notation depend only on the IFSs involved.
\begin{prop} \label{prop:bounded-distortion} Let $\mathcal{I}=\{f_{i}:i\in\Lambda\}$
be a regular IFS. Then for any finite word $a\in\Lambda^{*}$
and any $i,j\in\Lambda$, \[
\dist(I(a i),I(a j))=\Theta(|I(a)|),\]
 and \[
|I(a i)|=\Theta(|I(a)|).
\]
\end{prop}

The next step is to relate $\mu$ to its micromeasures. Recall that the family of micromeasures of a measure $\eta$ on $\RR^k$ is denoted $\left\langle \eta\right\rangle$. The following theorem is analogous to Proposition \ref{prop:CP-chain-for-self-similar-measure}:

\begin{prop} \label{prop:put-derived-measure-into-limit-geometry}
For any $\nu\in \left\langle \mu\right\rangle$ which is not supported on $\{ 0,1\}$, there is a limit diffeomorphism $F$ and an interval $J$ such that
\[
  F\mu \sim_{C} \nu^J
\]
\end{prop}
\begin{proof} By definition, there are intervals $I_{n}\subseteq[0,1]$ such that \[
\nu=\lim_{n\rightarrow\infty}\mu^{I_{n}},\]
 where $\mu^{I}=T_{I}(\mu_{I})$ is the rescaling of $\mu_{I}$
back to the unit interval. For each $n$ let $a^{(n)}$ be a minimal word such that $I(a^{(n)})\subseteq I_n$ (there could be several such words; pick one of them). Since $\nu$ is not supported on $\{0,1\}$, this is well-defined for large $n$, and moreover $|I_n|= O(|I(a^{(n)})|)$ by bounded distortion (i.e. Proposition \ref{prop:bounded-distortion}).

By Lemma \ref{lem:restrictions-of-quasi-product-measures}, writing $f_n = f_{a^{(n)}}$,
\[
 \mu^{I(a^{(n)})} \,\sim_{C}\, f_{n}^*\mu.
\]
On the other hand, for each $n$ there is an interval $J_n$ (corresponding to the relative position of $I(a^{(n)})$ inside $I_n$) such that $|J_n|=\Theta(1)$, and
\[
\mu^{I(a^{(n)})} = T_{J_n} ((\mu^{I_n})_{J_n}.)
\]
By passing to a subsequence we can assume that $T_{J_n}$ converges
to an affine map $T_J$, and $f_n^*$ to a limit diffeomorphism $F$. Taking weak limits yields the
proposition. \end{proof}

\subsection{Proof of Theorem \ref{thm:convolutions-of-quasi-product-measures}}

Throughout this section, $\mu_1,\mu_2$ are quasi-product measures associated to regular IFSs $\mathcal{I}_i= \{ f_j^{(i)} : j\in\Lambda_i\} $, $i=1,2$. We write $\mu = \mu_1\times \mu_2$ and
\[
\gamma = \min(1,\dim\mu).
\]
We say that $F:[0,1]^2\rightarrow [0,1]^2$ is a limit diffeomorphism if $F=F_1\times F_2$ where $F_i$ is a limit diffeomorphism for $\mathcal{I}_i$.

We next state the main three steps in the proof of Theorem \ref{thm:convolutions-of-quasi-product-measures}. The first one is, as usual, a topological version of the projection theorem, which does not require any minimality assumptions; compare with Proposition \ref{prop:semicontinuity-projections-of-self-similar-measures}.
Given a measure $\nu$ on $\RR^d$ and $\alpha\ge 0$, let
\begin{equation} \label{eq:def-good-C1-maps}
\mathcal{U}_\alpha(\nu) = \text{interior}\{ g\in C^1(\supp\nu,\RR^k) : \dim_*g\nu > \alpha \}.
\end{equation}
Note that these sets are by definition open, though a priori may be empty.

\begin{prop} \label{prop:semicontinuity-convolutions-of-quasi-product-measures}
For every $\e>0$, there exists a limit diffeomorphism $F$ such that
\[
\mathcal{U}_{\gamma-\e}(F\mu) \text{ contains a dense (and automatically open) subset of } \Pi_{2,1}.
\]
\end{prop}

Identify $\Pi_{2,1}\backslash\{\pi_x,\pi_y\}$ with $\mathbb{R}\backslash\{0\}$ via $s\rightarrow \pi_s(x,y)=x+sy$ (the latter is not technically an element of $\Pi_{2,1}$ but there is an obvious identification). This is similar but not the same as the identification in Section \ref{sec:Measures-invariant-under-xm}. Given $K\gg 1$, write
\[
I_K = [-K, -K^{-1} ] \cup [K^{-1},K].
\]
Using the irreducibility assumption to move the open and dense set around as usual (but with some extra technical complications), one obtains:
\begin{prop} \label{prop:convolutions-of-quasi-product-on-some-limit}
If additionally the minimality condition in Theorem \ref{thm:convolutions-of-Gibbs-measures} holds, then for every $\e>0$ and every $K>1$ there is a limit diffeomorphism $F$ such that
\[
\{ \pi_s : s\in I_K \} \subseteq \mathcal{U}_{\gamma-\e}(F\mu).
\]
\end{prop}

A priori, this holds for just one limit diffeomorphism $F=F_1\times F_2$. But this automatically implies it holds in fact for all limit diffeomorphisms:

\begin{prop} \label{prop:convolutions-of-quasi-product-on-all-limits}
Under the hypotheses of Proposition \ref{prop:convolutions-of-quasi-product-on-some-limit}, for every $\e>0$, every $K>1$, and \emph{every} limit diffeomorphism $F$,
\[
\{ \pi_s : s\in I_K \}  \subseteq \mathcal{U}_{\gamma-\e}(F\mu).
\]
\end{prop}

Before proving these propositions we use them to deduce Theorem \ref{thm:convolutions-of-quasi-product-measures}.

We say that a map $A:\RR^2\rightarrow\RR^2$ is \emph{affine diagonal} if it is affine with a diagonal linear part. In this case, we can write $A= H D$, where $H$ is an homothety and $D(x,y) = (x, a y)$ for some $a\in\RR$. We refer to $a$ as the \emph{eccentricity} of $A$.

\begin{lem} \label{lem:transformation-under-diagonal-map}
Let $\nu, \tau$ be two measures on $\RR^d$, and suppose that $A\nu \,\sim_C\, \tau_Q$ for some affine diagonal map $A$ and some set $Q$ with $\tau(Q)>0$. Then for any $\alpha$,
\begin{align*}
\mathcal{U}_\alpha(\nu) &\supseteq \{  g A: g\in \mathcal{U}_\alpha(\tau) \},\\
\mathcal{U}_\alpha(\nu)\cap \Pi_{2,1} &\supseteq \{  \pi_{s a} : \pi_s\in\mathcal{U}_\alpha(\tau)  \},
\end{align*}
where $a$ is the eccentricity of $A$.
\end{lem}
\begin{proof}
The first part is immediate. For the second, note that if $D(x,y)=(x,ay)$ then $\pi_s D=\pi_{as}$, and note that homotheties commute with linear maps and preserve dimension.
\end{proof}

\begin{proof}[Proof of Theorem \ref{thm:convolutions-of-quasi-product-measures} (assuming Proposition \ref{prop:convolutions-of-quasi-product-on-all-limits})]
By Corollary \ref{cor:proving-expected-dim}, it is enough to show that
\[
\dim_*(\pi(\mu_1\times\mu_2))\ge \gamma.
\]
Fix $\e>0$ and $K>1$ for the rest of the proof. For $i=1,2$ let $x^{(i)} \in \Lambda_i^{-\NN}$, and write $F_i = F_{x^{(i)}}$, $F=F_1\times F_2$.
Since, by Proposition \ref{prop:convolutions-of-quasi-product-on-all-limits},
\begin{equation} \label{eq:bounded-linear-maps-are-good-for-F}
\{ \pi_s : s\in I_K \} \subseteq  \mathcal{U}_{\gamma-\e}(F\mu),
\end{equation}
we have by compactness of $I_K$ and openness of $\mathcal{U}_{\gamma-\e}(F\mu)$ that the same is true if one replaces $F$ by a sufficiently close diffeomorphism. Since the convergence of $f_{y_n \ldots y_1}^*$ to $F_y$ is uniform, it follows that that for sufficiently long initial segments $a_1, a_2$ of $x^{(1)}, x^{(2)}$ (where the threshold length is also independent of $x^{(i)}$), if we write $f^* = f_{a_1}^* \times f_{a_2}^*$, then \eqref{eq:bounded-linear-maps-are-good-for-F} holds with $f^*\mu$ instead of $F\mu$.

By the principle of bounded distortion, we can choose the lengths of $a_i$ such that $I(a_1),I(a_2)$ have lengths which differ by a factor of $O(1)$, and still ensure that the length of both is bounded above independently of other parameters. Consider now the affine diagonal map $A = T_{I(a_1)}\times T_{I(a_2)}$, i.e. $A$ is such that $f^* = A f$, where $f= f_{a_1}\times f_{a_2}$. By construction the eccentricity of $A$ is $O(1)$. Since $A(f\mu) = f^*\mu$, we get from Lemma \ref{lem:transformation-under-diagonal-map} that
\[
\{ \pi_s :s\in I_{K/O(1)} \} \subseteq \mathcal{U}_{\gamma-\e}(f\mu).
\]
Recall that $f\mu \,\sim_C\, \mu_{I(a_1)\times I(a_2)}$ by the quasi-product property. We may cover $\supp(\mu)$ by rectangles of the form $I(a_1)\times I(a_2)$, with $a_1,a_2$ chosen to be long enough for the above to hold. Thus since, $\e$ and $K$ were arbitrary, the proof is complete.
\end{proof}

\subsection{Proofs of the remaining propositions}

\begin{proof}[Proof of Proposition \ref{prop:semicontinuity-convolutions-of-quasi-product-measures}]
The proof is analogous to the proof of Proposition \ref{prop:semicontinuity-projections-of-self-similar-measures}. We apply Theorem \ref{thm:furstenberg} to obtain an ergodic CP-chain $(\mu_{n},Q_{n})_{n=1}^{\infty}$
such that almost surely $\mu_{1}\in \left\langle \mu\right\rangle$ and $\dim\mu_{1}=\dim\mu$.
Hence, we can find a micromeasure $\nu$ such that $\dim\nu\ge\dim\mu$,
and the conclusion of Corollary \ref{cor:open-good-set-in-C1} applies
to $\nu$: given $\e>0$, the family $\mathcal{U}_{\gamma-\e}(\nu)$ has dense intersection with $\Pi_{2,1}$.

By compactness, $\nu=\nu_{1}\times\nu_{2}$, where $\nu_{i}\in\left\langle \mu_i\right\rangle$. By Proposition \ref{prop:put-derived-measure-into-limit-geometry}, we can find a limit diffeomorphism $F$, an affine diagonal map $A$, and a rectangle $Q$, such that $A\mu \,\sim_C\, \nu_Q$. The proposition then follows from Lemma \ref{lem:transformation-under-diagonal-map}. \end{proof}

For the proof of Proposition \ref{prop:convolutions-of-quasi-product-on-some-limit} we need the following lemma:
\begin{lem} \label{lem:dense-intersection-passes-to-limit}
Suppose $U\subseteq\mathbb{R}\setminus\{0\}$, and
\[
\{ \pi_s : s\in U \} \subseteq \mathcal{U}_\alpha(\mu).
\]
Then for any limit diffeomorphism $F$ there is a $t= t(F) = \Theta(1)$ such that
\[
\{ \pi_{t s} : s\in U \} \subseteq \mathcal{U}_\alpha(F\mu).
\]
\end{lem}
\begin{proof}
This is very similar to the proof of Theorem \ref{thm:convolutions-of-quasi-product-measures} above: one uses the fact that, up to affine rescaling and $C$-equivalence, the limit measure $F\mu$ is $C^1$-close to the restriction of $\mu$ to a small rectangle $Q$ of eccentricity $\Theta(1)$. Since $\mathcal{U}_\alpha(\mu_Q) \supseteq \mathcal{U}_\alpha(\mu)$, and the rescaling induced on $\Pi_{2,1}$ by pre-composition with $T_Q$ transforms $\pi_s$ to $\pi_{ts}$, where $t$ is the eccentricity of $Q$, the lemma follows by applying this procedure to a sequence of rectangles $Q_n$ of side length tending to $0$ and eccentricities converging to some $t=\Theta(1)$.
\end{proof}

\begin{proof}[Proof of Proposition \ref{prop:convolutions-of-quasi-product-on-some-limit}]
To begin, note that a priori we cannot ``move around'' the dense set given by Proposition \ref{prop:semicontinuity-convolutions-of-quasi-product-measures} as done in e.g. Proposition \ref{prop:almost-linear-projection-good}, since for this we need the action of an affine map, and the IFSs involved are a priori nonlinear. Instead we linearize the relevant map in each IFS by passing to a new limit geometry.

We first present the proof in the case $\lambda(f^{(1)}_{a_1})/\lambda(f^{(2)}_{a_2})\notin\mathbb{Q}$ for some $a_i\in\Lambda_i^*$. By iterating the IFSs and relabeling we may then assume that $\lambda(f^{(1)}_{1})/\lambda(f^{(2)}_{1})\notin\mathbb{Q}$

Let $F=F_1\times F_2$ be the limit diffeomorphism given by Proposition \ref{prop:semicontinuity-convolutions-of-quasi-product-measures}. The conjugated IFS $\mathcal{I}'_i=\{ F_i f^{(i)}_j F_i^{-1}:j\in\Lambda_i\}$ satisfies the same hypotheses as the original one. Let $x=(\ldots,1,1)\in\Lambda_i^{-\NN}$. By Corollary \ref{cor:consequences-Sullivan}(i), the first map in each limit IFS $(\mathcal{I}_i)_x$ is affine. But by Corollary \ref{cor:consequences-Sullivan}(ii), $(\mathcal{I}_i)_x = (\mathcal{I}'_i)_x$. On the other hand, by Lemma \ref{lem:dense-intersection-passes-to-limit}, the conclusion of Proposition \ref{prop:semicontinuity-convolutions-of-quasi-product-measures} is inherited by $(\mathcal{I}'_i)_x$. We have therefore shown that there is no loss of generality in assuming that the conjugated IFSs $\{ g^{(i)}_j = F_i f^{(i)}_j F_i^{-1} : j\in\Lambda_i\}$ are such that $g^{(1)}_1$, $g^{(2)}_1$ are affine. Moreover, we still have
\begin{equation} \label{eq:irrationality-after-linearizing}
\lambda\left(g^{(1)}_1\right)/\lambda\left(g^{(2)}_1\right) \notin\mathbb{Q}.
\end{equation}
Now we are ready to cover the whole of $I_K$ by using the action of $g^{(1)}_1\times g^{(2)}_1$ on $\mathcal{U}_{\gamma-\e}(F\mu)$.  By \eqref{eq:irrationality-after-linearizing}, the collection of eccentricities of the (affine diagonal) maps
\[
g_{n_1,n_2} := \left(g^{(1)}_1\right)^{n_1}\times \left(g^{(2)}_1\right)^{n_2}
\]
is dense in $(0,\infty)$. Also note that
\[
g_{n_1,n_2}^{-1} (F\mu)_{Q_{n_1,n_2}}\,\sim_C\, F\mu,
\]
where $Q_{n_1,n_2} =g_{n_1,n_2}([0,1]^2)$. It then follows from Lemma \ref{lem:transformation-under-diagonal-map} that there is $N$ such that
\[
\{  \pi_s: s\in I_K\} \subseteq \bigcup_{n_1,n_2=1}^N \mathcal{U}_{\gamma-\e}((F\mu)_{Q_{n_1,n_2}}) .
\]
Pick $N_1, N_2\ge N$ such that the eccentricity of $g_{N_1, N_2}$ is less than, say, $2$. By the above, $\mathcal{U}_{\gamma-\e}\left((F\mu)_{Q_{N_1,N_2}}\right)$ contains $\{ \pi_s:s\in I_K\}$. But then using Lemma \ref{lem:transformation-under-diagonal-map} again we conclude that $\mathcal{U}_{\gamma-\e}(F\mu)$ contains $\{ \pi_s:s\in I_{K/2}\}$, as desired.

In the general case we can still, by minimality and Proposition \ref{prop:semicontinuity-convolutions-of-quasi-product-measures}, find $a_i\in\Lambda_i^*$, $i=1,2$, and an $N$ such that
\[
I_K \subseteq \bigcup_{n_1, n_2=1}^N \left\{ \lambda\left(f^{(1)}_{a_1}\right) ^{n_1} \lambda\left(f^{(2)}_{a_2}\right)^{-n_2} t: \pi_t\in \mathcal{U}_{\gamma-\e}(F\mu)\right\}.
\]
By Lemma \ref{lem:dense-intersection-passes-to-limit} the same holds after passing to the limit IFSs which linearize the maps $f^{(i)}_{a_i}$, at the cost of replacing $I_K$ by $I_{\Omega(K)}$. The argument then proceeds in the same manner as before. \end{proof}

\begin{proof}[Proof of Proposition \ref{prop:convolutions-of-quasi-product-on-all-limits}]
By Lemma \ref{lem:dense-intersection-passes-to-limit}, if $\mathcal{U}_\alpha(\mu)$ contains $\{ \pi_s :s\in I_K\}$, then the same is true with $K$ replaced by $\Omega(K)$ and $\mu$ replaced by $F\mu$ for any limit diffeomorphism $F$. We know that $\mathcal{U}_\alpha(G\mu)$ contains $I_K$ for some limit diffeomorphism $G$. Thus Corollary \ref{cor:consequences-Sullivan}(ii) and the previous observation applied to the conjugated IFSs $\{ G_i f^{(i)}_j G_i^{-1} :i\in\Lambda_i \}$ (i.e. to the measure $G\mu$) yield the result.
\end{proof}

\noindent{\bf Acknowledgement}.
This project began during the program ``Ergodic theory and additive combinatorics'' at MSRI, and we are grateful to the organizers and MSRI for the stimulating atmosphere. We also thank Yuval Peres for useful comments, and particularly for suggesting a simplification to the proof of Lemma \ref{lem:local-entropy-lemma}.

\bibliographystyle{plain}
\bibliography{dimofprojections}

\end{document}